%% file: RepsofSp6andRestrictionstoMaxSubgps-PaperVersion.tex
\documentclass[11pt]{article}
\usepackage{url}
\usepackage{amsfonts}
\usepackage{graphics} 
\usepackage[pdftex]{graphicx}
\usepackage{amsopn} 
\usepackage[dvips]{epsfig} 
\usepackage{amsxtra}
\usepackage{amssymb}
\usepackage{amsmath}
\usepackage[version=3]{mhchem}
\usepackage{amsthm}
\usepackage[lofdepth,lotdepth]{subfig}
\usepackage{setspace,amssymb,verbatim}
\usepackage[margin=1in]{geometry}
\usepackage{cmap}
\usepackage{mathdots}
\usepackage{breakcites}
\usepackage{microtype}
\usepackage[pdftex]{color}
\usepackage[pdftex]{graphicx}
\usepackage[noadjust]{cite}
\usepackage{tikz}
\usepackage[all]{xy}
\usepackage{csquotes}
\usepackage{prettyref}
\newrefformat{sec}{Section \ref{#1}}
\newrefformat{prop}{Proposition \ref{#1}}
\newrefformat{prob}{Problem \ref{#1}}
\newrefformat{rem}{Remark \ref{#1}}
\newrefformat{cor}{Corollary \ref{#1}}
\newrefformat{def}{Definition \ref{#1}}
\newrefformat{ex}{Example \ref{#1}}
\newrefformat{lem}{Lemma \ref{#1}}
\newrefformat{eq}{Equation \eqref{#1}}
\newrefformat{con}{condition (\ref{#1})}
\newrefformat{conj}{Conjecture \ref{#1}}
\newrefformat{fig}{Figure \ref{#1}}
\newrefformat{tab}{Table \ref{#1}}
\newrefformat{app}{Appendix \ref{#1}}

\theoremstyle{plain}
\newtheorem{theorem}{Theorem}[section]
\newtheorem{proposition}[theorem]{Proposition}
\newtheorem*{theorem*}{Theorem}
\newtheorem{corollary}[theorem]{Corollary}
\newtheorem*{lemma*}{Lemma}
\newtheorem{lemma}[theorem]{Lemma}

\theoremstyle{remark}

\newtheorem*{case*}{Case}

\theoremstyle{definition}

\newtheorem{problem}{Problem}

\newcommand{\tw}[1]{{}^#1\!}
\newcommand{\stab}{\mathrm{stab}}

\newcommand{\Z}{\mathbb{Z}}

\newcommand{\C}{\mathbb{C}}

\newcommand{\F}{\mathbb{F}}

\newcommand{\irr}{\mathrm{Irr}}
\newcommand{\ibr}{\mathrm{IBr}}

\newcommand{\wh}[1]{\widehat{#1}}
\newcommand{\tr}{\mathrm{Tr}}

\title{Cross-Characteristic Representations of $Sp_6(2^a)$ and Their Restrictions to Maximal Subgroups}

\author{Amanda A. Schaeffer Fry\\\small\textit{Department of Mathematics, University of Arizona, Tucson, AZ 85721, USA}}
\date{}
\begin{document}
\maketitle

\begin{abstract}
We classify all pairs $(V,H)$, where $H$ is a proper subgroup of $G=Sp_6(q)$, $q$ even, and $V$ is an $\ell$-modular representation of $G$ for $\ell\neq 2$ which is absolutely irreducible as a representation of $H$.  This problem is motivated by the Aschbacher-Scott program on classifying maximal subgroups of finite classical groups.

\end{abstract}

\section{Introduction}

\input{PAPERVERSIONSp6restrictionsintro.tex}
\section{Some Preliminary Observations}\label{sec:prelimobservations}
\input{PAPERVERSIONpreliminaryobservationsSp6restrictions.tex}

\section{Low-Dimensional Representations of $Sp_6(q)$}\label{sec:lowdimreps}
\input{PAPERVERSIONlowdimrepsSp6.tex}
\section{A Basic Reduction}\label{sec:BasicReduction}
\input{PAPERVERSIONCharsRestSp6evencharBasicreduction.tex}
\section{Restrictions of Irreducible Characters of $Sp_6(q)$ to $G_2(q)$}\label{sec:resttoG2}

\input{PAPERVERSIONCharsRestSp6G2evenchar.tex}

\section{Restrictions of Irreducible Characters of $Sp_6(q)$  to the Subgroups $O_6^{\pm}(q)$}\label{sec:resttoO6}
\input{PAPERVERSIONCharsRestSp6O6evenchar.tex}
\section{Restrictions of Irreducible Characters to Maximal Parabolic Subgroups}\label{sec:resttoparabolics}
\input{PAPERVERSIONCharsRestSp6Parabolicsevenchar.tex}

\section{The case $q=2$}\label{sec:caseqis2}
In this section, we prove Theorems \ref{thm:mainresultSp4(2)} and \ref{thm:mainresultSp6(2)}.  To do this, we use the computer algebra system GAP, \cite{GAP4}.  In particular, we utilize the character table library \cite{GAPctlib}, in which the ordinary and Brauer character tables for $Sp_6(2)$ and $S_4(2)\cong S_6$, along with all of their maximal subgroups, are stored.  The maximal subgroups of $Sp_6(2)$ are as follows:
 \[U_4(2).2, \quad A_8.2, \quad 2^5:S_6, \quad U_3(3).2, \quad 2^6:L_3(2),\quad 2.[2^6]:(S_3\times S_3), \quad S_3\times S_6, \quad L_2(8).3,\] and the maximal subgroups of $Sp_4(2)\cong S_6$ are
 \[ A_6,\quad A_5.2=S_5, \quad O_4^-(2)\cong S_5,\quad S_3\wr S_2,\quad 2\times S_4, \quad S_2\wr S_3\]
The ordinary and Brauer character tables for each of these maximal subgroups are stored as well, with the exception of $2^5:S_6$ and $2^6:L_3(2)$, for which we only have the ordinary character tables.  In addition, the command \verb+PossibleClassFusions(c1,c2)+ gives all possible fusions from the group whose (Brauer) character table is \verb+c1+ and the group whose (Brauer) character table is \verb+c2+.  Using this command, it is straightforward to find all Brauer characters which restrict irreducibly from \verb+c2+ to \verb+c1+.

In the case $H=P_3=2^6:L_3(2)$ or $2^5: S_6$ and $G=Sp_6(2)$, we need additional techniques, as the Brauer character tables for these choices of $H$ are not stored in the GAP character table library.  However, in the case $H=2^5: S_6$, the above technique shows that there are no ordinary irreducible characters of $G$ which restrict irreducibly to $H$, and moreover, there is no $\chi|_H-\lambda$ for $\chi\in\irr(G), \lambda\in\wh{H}$ which is irreducible on $H$.  Observing that any $\varphi\in\ibr_\ell(G)$ with $\varphi(1)\leq \mathfrak{m}(H)$ either lifts to a complex character or is $\wh{\chi}-1$ for some complex character $\chi$, we see by \prettyref{lem:Lemma1forO6} that there are no irreducible Brauer characters of $G$ which restrict to an irreducible Brauer character of $H$, for any choice of $\ell\neq 2$.

We are therefore reduced to the case $H=P_3$.  In this case, it is clear from our above techniques that the only ordinary characters which restrict irreducibly to $H$ are $\alpha_3$ and $\chi_4$, where $\chi_4$ is the irreducible character of degree $21$ which is not $\zeta_3^1$.  Moreover, there is again no $\chi\in\irr(G)$, $\lambda\in\wh{H}$ such that $\chi|_H-\lambda\in\irr(H)$.  Referring to the notation of \prettyref{sec:resttoparabolics}, we have $|\mathcal{O}_1|=7$, $|\mathcal{O}_2^-|=7, |\mathcal{O}_2^+|=14,$ and $|\mathcal{O}_3|=28$.  Hence, any $\varphi\in\ibr_\ell(H)$ must satisfy $7|\varphi(1)$.  We can see from the Brauer character table of $G$ that if $\chi\in\ibr_\ell(G)$ has $\chi(1)\leq \mathfrak{m}(H)$, then either $\chi$ or $\chi+1$ lifts to $\C$.  Thus by \prettyref{lem:Lemma1forO6}, the only possibilities are $\wh{\alpha}_3$ and $\wh{\chi_4}$.  Now $\wh{\alpha}_3|_H$ is irreducible since $\alpha_3(1)=|\mathcal{O}_1|=|\mathcal{O}_2^-|$, and these are the smallest orbits of characters in $Z_3$.  Therefore it remains only to show that $\wh{\chi}_4$ is indeed also irreducible on $H$.

Since we know that $\chi_4|_H\in\irr(H),$ we know that $\chi_4|_{Z_3}$ must contain only one orbit of $Z_3$-characters as constituents, which means that $\chi_4|_{Z_3}=3\omega_1$ or $3\omega_2^-$, continuing with the notation of \prettyref{sec:resttoparabolics}.  Since $Z_3$ consists of $2$-elements, we know $\wh{\chi_4}|_{Z_3}$ can be written in the same way.  Moreover, since $q=2$, $\stab_{L_3}(\lambda)$ is solvable for $\lambda\neq 1$, so we know that if $\lambda$ is a constituent of $\chi_4|_{Z_3}$, then any $\psi\in\ibr_\ell(I|\lambda)$ lifts to an ordinary character.  Since by Clifford theory, any irreducible constituent of $\wh{\chi}_4|_H$ can be written $\psi^H$ for such a $\psi$, it follows that if $\wh{\chi}_4|_H$ is reducible, then it can be written as the sum of some $\wh{\varphi}_i$ for $\varphi_i\in\irr(H|\lambda)$.  In particular, each of these $\varphi_i$ must have degree $7$ or $14$.  By inspection of the columns of the ordinary character table of $H$ corresponding to $3$-regular and $7$-regular classes, it is clear that $\chi_4|_H$ cannot be written as such a sum on $\ell$-regular elements, and therefore $\wh{\chi}_4|_H$ is irreducible.

\section{Acknowledgments}
This work was supported in part by a graduate fellowship from the National Physical Science Consortium and by the NSF (grant DMS-0901241) and was completed as part of my Ph.D. study at the University of Arizona.  I would like to thank my advisor, Professor P.H. Tiep, for all of his support, guidance, and helpful discussions.  I would also like to thank Professor K. Lux for numerous helpful discussions and for introducing me to computations with GAP, and Professor F. Himstedt of Technische Universit{\"a}t M{\"u}nchen for helping me get started with CHEVIE.

\appendix
\section{The Brauer Characters for $Sp_6(2^a)$ Lying in Unipotent Blocks}\label{app:BrauerCharTable}
\input{PAPERVERSIONBrauerCharTableSp6q.tex}

\section{Notations of Characters in $Sp_6(q)$ and $G_2(q)$}\label{app:notationschars}
\input{PAPERVERSIONNotationsCharsSpG2.tex}
\prettyref{tab:notationSp} and \prettyref{tab:notationG2} give the notation from different authors for the characters of $Sp_6(q)$ and $G_2(q)$ we most frequently refer to.
\clearpage
\bibliographystyle{unsrt}
\bibliography{researchreferences}
\end{document}

%% file: PAPERVERSIONSp6restrictionsintro.tex
Finite primitive permutation groups have been a topic of interest since the time of Galois and have applications to many areas of mathematics.  A transitive permutation group $X\leq \rm{Sym}(\Omega)$ is primitive if and only if any point stabilizer $H=\stab_X(\alpha)$, for $\alpha\in \Omega$, is a maximal subgroup.  Many problems involving such groups can be reduced to the special case where $X$ is a finite classical group.  In this case, Aschbacher has described all possible choices for the maximal subgroup $H$ (see \cite{aschbacher}).  Namely, he has described 8 collections $\mathcal{C}_1,...,\mathcal{C}_8$ of subgroups obtained in natural ways (for example, stabilizers of certain subspaces of the natural module for $X$), and a collection $\mathcal{S}$ of almost quasi-simple groups which act absolutely irreducibly on the natural module for $X$.  The question of whether a subgroup $H$ in $\bigcup_{i=1}^8 \mathcal{C}_i$ is in fact maximal has been answered by Kleidman and Liebeck, (see \cite{KleidmanLiebeck}).  When $H\in \mathcal{S}$, we want to decide whether there is some maximal subgroup $G$ such that $H<G<X$, that is, if $H$ is not maximal.  The most challenging case is when $G$ also lies in the collection $\mathcal{S}$.  This suggests the following problem, which is the motivation for this paper.

\begin{problem}\label{prob:aschbakerscottprogram}
Let $\mathbb{F}$ be an algebraically closed field of characteristic $\ell\geq0$.  Classify all triples $(G,V,H)$ where $G$ is a finite group with $G/Z(G)$ almost simple, $V$ is an $\mathbb{F}G-$module of dimension greater than $1$, and $H$ is a proper subgroup of $G$ such that the restriction $V|_H$ is irreducible.
\end{problem}

In \cite{brundankleshchev01}, \cite{kleshchevsheth02}, and \cite{kleshchevtiep04}, Brundan, Kleshchev, Sheth, and Tiep have solved \prettyref{prob:aschbakerscottprogram} for $\ell>3$ when $G/Z(G)$ is an alternating or symmetric group.  Liebeck, Seitz, and Testerman have obtained results for Lie-type groups in defining characteristic $\ell$ in \cite{liebeck85}, \cite{seitz87}, and \cite{seitztesterman90}.

Assume now that $G$ is a finite group of Lie type defined in characteristic $p\neq \ell$, with $q$ a power of $p$.  In \cite{nguyentiep08}, Nguyen and Tiep show that when $G=\tw3 D_4(q)$, the restrictions of irreducible representations are reducible over every proper subgroup, and in \cite{himstedtnguyentiep09}, Himstedt, Nguyen, and Tiep prove that this is the case for $G=\tw2 F_4(q)$ as well.  Nguyen shows in \cite{Nguyen08} that when $G=G_2(q), \tw2 G_2(q),$ or $\tw2 B_2(q)$, there are examples of triples as in \prettyref{prob:aschbakerscottprogram} and finds all such examples.  In \cite{TiepKleshchev10}, Tiep and Kleshchev solve \prettyref{prob:aschbakerscottprogram} in the case that $SL_n(q)\leq G\leq GL_n(q)$.  In \cite{seitz}, Seitz provides a list of possibilities for $(H,G)$ as in \prettyref{prob:aschbakerscottprogram} in the case that $H$ is a finite group of Lie type and $G$ is a finite classical group, both defined in the same characteristic.  In particular, his results signify the importance of studying \prettyref{prob:aschbakerscottprogram} in the case $G=Sp_6(2^a)$.

Here we focus on the case where $G=Sp_{2n}(q)$ for $n=2,3$ with $q$ even, and $H$ is a proper subgroup.  In considering this problem, it is useful to know the low-dimensional $\ell$-modular representations of $Sp_6(q)$.  We prove the following theorem, which describes these representations.  In the theorem, let $\alpha_3, \beta_3, \rho_3^1, \rho_3^2, \tau_3^i,$ and $\zeta_3^i$ denote the complex Weil characters of $Sp_6(q)$, as in \cite{TiepGuralnick04} (see \prettyref{tab:TiepGuralnickTable1}), and let $\chi_j$, $1\leq j\leq 35$ be as in the notation of \cite{white2000}.

\begin{theorem}\label{thm:lowdimreps}
Let $G=Sp_6(q)$, with $q\geq 4$ even, and let $\ell\neq 2$ be a prime dividing $|G|$.  Suppose $\chi\in\ibr_\ell(G)$.  Then:

A) If $\chi$ lies in a unipotent $\ell$-block, then either
\begin{enumerate}
\item $\chi\in\left\{1_G,  \wh{\alpha}_3,  \wh{\rho}_3^1- \left\{\begin{array}{cc}
                                                                             1, & \ell|(q^2+q+1),\\
                                                                             0, & \hbox{otherwise}
                                                                           \end{array}\right.,\quad                                                                      \widehat{\beta}_3- \left\{\begin{array}{cc}
                                                                             1, & \ell|(q+1),\\
                                                                             0, & \hbox{otherwise}
                                                                           \end{array}\right.,\quad
                                                                            \widehat{\rho}_3^2 - \left\{\begin{array}{cc}
                                                                             1, & \ell|(q^3+1),\\
                                                                             0, & \hbox{otherwise}
                                                                           \end{array}\right.\right\}$,
\item $\chi$ is as in the following table:
\begin{center}\begin{tabular}{|c|c|c|}
                                             \hline
Condition on $\ell$ & $\chi$ & Degree $\chi(1)$ \\
\hline
$\ell|(q^3-1)$ or & & \\
 $3\neq\ell|(q^2-q+1)$ & $\wh{\chi}_6$ & $q^2(q^4+q^2+1)$\\
\hline
$\ell|(q^2+1)$ & $\wh{\chi}_6-1_G $ & $q^2(q^4+q^2+1)-1$\\
\hline
 & $\wh{\chi}_{28}$ & \\
$\ell|(q+1)$ & $=\wh{\chi}_6-\wh{\chi}_3-\wh{\chi}_2+1_G$ & $(q^2+q+1)(q-1)^2(q^2+1)$ \\
                                             \hline
                                           \end{tabular},\end{center}

\item $\chi$ is as in the following table:
\begin{center}\begin{tabular}{|c|c|c|}
                                             \hline
Condition on $\ell$ & $\chi$ & Degree $\chi(1)$ \\
\hline
$\ell|(q^3-1)$ or & & \\
 $3\neq\ell|(q^2-q+1)$ & $\wh{\chi}_7$ & $q^3(q^4+q^2+1)$\\
\hline
$\ell|(q^2+1)$ & $\wh{\chi}_7-\wh{\chi}_4$ & $q^3(q^4+q^2+1)-q(q+1)(q^3+1)/2 $\\
\hline
 & $\wh{\chi}_{35}-\wh{\chi}_5$ & \\
$\ell|(q+1)$ & $=\wh{\chi}_7-\wh{\chi}_6+\wh{\chi}_3-\wh{\chi}_1$ & $(q-1)(q^2+1)(q^4+q^2+1)-q(q-1)(q^3-1)/2$ \\
                                             \hline
                                           \end{tabular},\end{center}
                                            or
\item $\chi(1)\geq D$, where $D$ is as in the table:
\begin{center}
\begin{tabular}{|c|c|}
\hline
Condition on $\ell$ & $D$\\
\hline
$\ell|(q^3-1)(q^2+1)$ & $\frac{1}{2}q^4(q-1)^2(q^2+q+1)$\\
\hline
$\ell|(q+1)$,  &\\
$(q+1)_\ell\neq3$ & $\frac{1}{2}q (q^3 - 2) (q^2 + 1) (q^2 - q + 1) - \frac{1}{2}q (q - 1) (q^3 - 1) + 1$\\
\hline
$\ell|(q+1)$,  &\\
$(q+1)_\ell=3$ &$\frac{1}{2}q (q^3 - 2) (q^2 + 1) (q^2 - q + 1)  + 1$ \\
\hline
$3\neq\ell|(q^2-q+1)$ & $\frac{1}{2}q^4(q-1)^2(q^2+q+1)-\frac{1}{2}q(q-1)^2(q^2+q+1)=\frac{1}{2}q(q^3-1)^2(q-1)$ \\
\hline
\end{tabular}
\end{center}
\end{enumerate}
B) If $\chi$ does not lie in a unipotent block, then either
\begin{enumerate}
\item $\chi\in\{\wh{\tau}_3^i,  \wh{\zeta}_3^j|  1\leq i\leq ((q-1)_{\ell'}-1)/2,  1\leq j\leq ((q+1)_{\ell'}-1)/2\}$,
\item $\chi(1)=(q^2+1)(q-1)^2(q^2+q+1)$ or $(q^2+1)(q+1)^2(q^2-q+1)$ (here $\chi$ is the restriction to $\ell$-regular elements of the semisimple character indexed by a semisimple $\ell'$ - class in the family  $c_{6,0}$ or $c_{5,0}$ respectively),
\item $\chi(1)=(q-1)(q^2+1)(q^4+q^2+1)$ or $(q+1)(q^2+1)(q^4+q^2+1)$ (here $\chi$ is the restriction to $\ell$-regular elements of the semisimple character indexed by a semisimple $\ell'$ - class in the family $c_{10,0}$ or $c_{8,0}$ respectively), or
\item $\chi(1)\geq q(q^4+q^2+1)(q-1)^3/2$.
\end{enumerate}
\end{theorem}
Note that \prettyref{thm:lowdimreps} generalizes \cite[Theorem 6.1]{TiepGuralnick04}, which gives the corresponding bounds for ordinary representations of $Sp_{2n}(q)$ with $q$ even.

Our main result is the following complete classification of triples $(G,V,H)$ as in \prettyref{prob:aschbakerscottprogram} in the case $G=Sp_{6}(q)$ with $q\geq 4$ even.
\begin{theorem}\label{thm:mainresult}
Let $q$ be a power of $2$ larger than $2$, and let $(G,V,H)$ be a triple as in \prettyref{prob:aschbakerscottprogram}, with $\ell\neq 2$, $G=Sp_{6}(q)$, and $H<G$ a proper subgroup.  Then:
\begin{enumerate}
\item $P_3'\leq H\leq P_3$, the stabilizer of a totally singular $3$-dimensional subspace of the natural module $\F_q^{6}$, and the Brauer character afforded by $V$ is the Weil character $\widehat{\alpha_3}$; or
\item $H=G_2(q)$, and the Brauer character afforded by $V$ is one of the Weil characters \begin{itemize}
    \item $\widehat{\rho}_3^1 - \left\{\begin{array}{cc}
                                1, & \ell|\frac{q^3-1}{q-1},\\
                                0, & \hbox{otherwise}
                                \end{array}\right.$, \quad degree $q(q+1)(q^3+1)/2-\left\{\begin{array}{c}
                                1 \\
                                0
                                \end{array}\right.$
    \item $\widehat{\tau}_3^i$, $1\leq i\leq ((q-1)_{\ell'}-1)/2$, \quad degree $(q^6-1)/(q-1)$
    \item $ \widehat{\alpha}_3$, \quad degree $q(q-1)(q^3-1)/2$
    \item $\widehat{\zeta}_3^i$, $1\leq i\leq ((q+1)_{\ell'}-1)/2$, \quad degree $(q^6-1)/(q+1)$.
\end{itemize}
as in the notation of \cite{TiepGuralnick04} (see \prettyref{tab:TiepGuralnickTable1}).
\end{enumerate}
Moreover, each of the above situations indeed gives rise to such a triple $(G,V,H)$.
\end{theorem}

Note that \prettyref{thm:mainresult} tells us that pair (ii) in the main theorem of \cite{seitz} does not occur for the case $n=7$, $q$ even, and that pair (iv) does occur.

We also prove the following complete classifications of triples as in \prettyref{prob:aschbakerscottprogram} when $H$ is a maximal subgroup of $G=Sp_{4}(q)$, $q\geq 4$ even, $G=Sp_6(2)$, and $G=Sp_4(2)$.

\begin{theorem}\label{thm:mainresultSp4}
Let $q$ be a power of $2$ larger than $2$, $\ell\neq 2$, $G=Sp_{4}(q)$, and $H<G$ a maximal subgroup.  Then $(G,V,H)$ is a triple as in \prettyref{prob:aschbakerscottprogram} if and only if $H=P_2$, the stabilizer of a totally singular $2$-dimensional subspace of the natural module $\F_q^{4}$, and the Brauer character afforded by $V$ is the Weil character $\widehat{\alpha_2}$.
\end{theorem}

\begin{theorem}\label{thm:mainresultSp4(2)}
Let $(G,V,H)$ be a triple as in \prettyref{prob:aschbakerscottprogram}, with $\ell\neq 2$, $G=Sp_{4}(2)\cong S_6$, and $H<G$ a maximal subgroup.  Then one of the following situations holds:
\begin{enumerate}
\item $H=A_6$, 
\item $H=A_5.2=S_5$, 
\item $H=O_4^-(2)\cong S_5=A_6.2_1M3$ in the notation of \cite{GAPctlib}.
\end{enumerate}
Moreover, each of the above situations indeed gives rise to such a triple $(G,V,H)$.  

\end{theorem}

\begin{theorem}\label{thm:mainresultSp6(2)}
Let $(G,V,H)$ be a triple as in \prettyref{prob:aschbakerscottprogram}, with $\ell\neq 2$, $G=Sp_{6}(2)$, and $H<G$ a maximal subgroup.  Then one of the following situations holds:
\begin{enumerate}
\item $H=G_2(2)=U_3(3).2$, and
\begin{itemize}
\item $\ell=0,5,7$ and $V$ affords the Brauer character $\wh{\alpha}_3$,  $\wh{\zeta}_3^1$,  $\wh{\rho}_3^1-\left\{\begin{array}{cc} 1, & \ell=7\\ 0, & otherwise\end{array}\right.,$ or $\wh{\chi}_9$, where $\chi_9$ is the unique irreducible complex character of $Sp_6(2)$ of degree $56$.
\item $\ell=3$ and $V$ affords the Brauer character $\wh{\alpha}_3$ or $\wh{\rho_3}^1$.
\end{itemize}
\item  $H=O_6^-(2)\cong U_4(2).2$, and the Brauer character afforded by $V$ is the Weil character $\widehat{\beta_3}$.
\item  $H=O_6^+(2)\cong L_4(2).2\cong A_8.2$, and the Brauer character afforded by $V$ is either the Weil character $\wh{\alpha_3}$, the character $\wh{\chi}_7$ where $\chi_7$ is the unique irreducible character of degree $35$ which is not equal to $\rho_3^2$, or the character $\wh{\chi_4}$ where $\chi_4$ is the unique irreducible character of degree $21$ which is not equal to $\zeta_3^1$.
\item  $H=2^6:L_3(2)$, and the Brauer character afforded by $V$ is $\wh{\alpha}_3$ or $\wh{\chi_4}$ where $\chi_4$ is the unique irreducible character of $G$ of degree $21$ which is not equal to $\zeta_3^1$.
\item  $H=L_2(8).3$, and $V$ affords one of the Brauer characters:
\begin{itemize}
\item $\wh{\alpha}_3$,
\item $\wh{\zeta}_3^1$, $\quad\ell\neq3,$
 \item $\wh{\rho}_3^1,\quad \ell\neq 7,$
 or\item $\wh{\chi_4}$ where $\chi_4$ is the unique irreducible complex character of $Sp_6(2)$ of degree $21$ which is not equal to $\zeta_3^1$, $\quad \ell\neq 3$.
\end{itemize}
\end{enumerate}
Moreover, each of the above situations indeed gives rise to such a triple $(G,V,H)$.
\end{theorem}

We note that unlike the case $q\geq 4$, we do not discuss the descent to non-maximal proper subgroups of $Sp_6(2)$ in \prettyref{thm:mainresultSp6(2)}, as there are many examples of such triples in this case.

We begin in \prettyref{sec:prelimobservations} by making some preliminary observations and listing some useful facts before proving \prettyref{thm:lowdimreps} in \prettyref{sec:lowdimreps}. In the remaining sections, we prove \prettyref{thm:mainresult} and \prettyref{thm:mainresultSp4}, first making a basic reduction to rule out a few subgroups, then treating each remaining maximal subgroup $H$ separately to find all irreducible $G-$modules $V$ which restrict irreducibly to $H$.  Finally, in \prettyref{sec:caseqis2} we treat the case $q=2$ and prove Theorems \ref{thm:mainresultSp4(2)} and \ref{thm:mainresultSp6(2)}. 

%% file: PAPERVERSIONpreliminaryobservationsSp6restrictions.tex
We adapt the notation of \cite{KleidmanLiebeck} for the finite groups of Lie type.  In particular, $L_n(q)$ and $U_n(q)$ will denote the groups $PSL_n(q)$ and $PSU_n(q)$, respectively.  $O^+_{2n}(q)$ and $O^-_{2n}(q)$ will denote the general orthogonal groups corresponding to quadratic forms of Witt defect 0 and 1, respectively.

Given a finite group $X$, we denote by $\mathfrak{d}_\ell(X)$ the smallest degree larger than one of absolutely irreducible representations of $X$ in characteristic $\ell$.  Similarly, $\mathfrak{m}_\ell(X)$ denotes the largest such degree.  When $\ell=0$, we write $\mathfrak{m}_0(X)=:\mathfrak{m}(X)$.   Given $\chi$ a complex character of $X$, we denote by $\widehat{\chi}$ the restriction of $\chi$ to $\ell$-regular elements of $X$, and we will say a Brauer character $\varphi$ lifts if $\varphi=\wh{\chi}$ for some complex character $\chi$.  Throughout the paper, $\ell$ will usually denote the characteristic of the representation.

As usual, $\irr(X)$ will denote the set of irreducible ordinary characters of $X$ and $\ibr_\ell(X)$ will denote the set of irreducible $\ell$-Brauer characters of $X$.  Given a subgroup $Y$ and a character $\lambda\in\ibr_\ell(Y)$, we will use $\ibr_\ell(X|\lambda)$ to denote the set of irreducible Brauer characters of $X$ which contain $\lambda$ as a constituent when restricted to $Y$.  The restriction of the Brauer character $\varphi$ to $Y$ will be written $\varphi|_Y$, and the induction of $\lambda$ to $X$ will be written $\lambda^X$.  We will use the notation $\ker\varphi$ to denote the kernel of the representation affording $\varphi\in\ibr_\ell(X)$.

We begin by making a few general observations, which we will sometimes use without reference:

\begin{lemma}\label{lem:maxmindegrees}
Let $G$ be a finite group, $H< G$ a proper subgroup, $\F$ an algebraically closed field of characteristic $\ell\geq0$, and $V$ an irreducible $\F G$-module with dimension greater than $1$.  Further, suppose that the restriction $V|_H$ is irreducible.  Then
\[\sqrt{|H/Z(H)|}\geq \mathfrak{m}(H)\geq\mathfrak{m}_\ell(H)\geq \dim(V)\geq\mathfrak{d}_\ell(G).\]
\end{lemma}

\begin{lemma}\label{lem:ifacharacterlifts}
Let $\chi\in\irr(G)$ such that $\wh{\chi}|_H\in\ibr_\ell(H)$.  Then $\chi|_H\in\irr(H)$.
\end{lemma}

\begin{lemma}\label{lem:Lemma1forO6}
Let $G$ be a finite group, $H\leq G$ a subgroup, and $\ell$ a prime.  Let $\widehat{H}$ denote the set of irreducible complex characters of degree $1$ of $H$.  If $\chi\in\irr(G)$ such that $\chi|_H-\lambda\not\in \irr(H)$ for any $\lambda\in\widehat{H}\cup\{0\}$, then $\widehat{\chi}|_H-\mu\not\in\ibr_\ell(H)$ for any $\mu\in\ibr_\ell(H)$ of degree $1$.
\end{lemma}

Lemmas \ref{lem:ifacharacterlifts} and \ref{lem:Lemma1forO6} suggest that in some situations, we will be able to reduce to the case of ordinary representations.

\subsection{Some Relevant Deligne-Lusztig Theory}

Let $G=\underline{G}^F$ for a connected reductive algebraic group $\underline{G}$ in characteristic $p\neq\ell$ and a Frobenius map $F$, and write $G^\ast=(\underline{G}^\ast)^{F^\ast}$, where $(\underline{G}^\ast, F^\ast)$ is dual to $(\underline{G}, F)$.   We can write $\irr(G)$ as a disjoint union $\bigsqcup \mathcal{E}(G, (s))$ of rational Lusztig series corresponding to $G^\ast$- conjugacy classes of semisimple elements $s\in G^\ast$.  Recall that the characters in the series $\mathcal{E}(G, (1))$ are called unipotent characters, and there is a bijection $\mathcal{E}(G,(s))\leftrightarrow\mathcal{E}(C_{G^\ast}(s), (1))$ such that if $\chi\mapsto \psi$, then $\chi(1)=[G^\ast:C_{G^\ast}(s)]_{p'}\psi(1)$.

Let $t$ be a semisimple $\ell'$ - element of $G^\ast$ and write $\mathcal{E}_\ell(G,(t)):=\bigcup\mathcal{E}(G,(ut)),$ where the union is taken over all $\ell$-elements $u$ in $C_{G^\ast}(t)$.  By a fundamental result of Brou{\'e} and Michel \cite{brouemichel}, $\mathcal{E}_\ell(G,(t))$ is a union of $\ell$-blocks.  Hence, we may view $\mathcal{E}_\ell(G,(t))$ as a collection of $\ell$-Brauer characters as well as a set of ordinary characters.

Moreover, it follows (see, for example \cite[Proposition 1]{hissmalle}) that the degree of any irreducible Brauer character $\varphi\in\mathcal{E}_\ell(G,(t))$ is divisible by $[G^\ast:C_{G^\ast}(t)]_{p'}$.  Hence, if $\chi\in\mathcal{E}_\ell(G,(t))\cap\irr(G)$ and $\chi(1)=[G^\ast:C_{G^\ast}(t)]_{p'}$, then $\wh{\chi}$ is irreducible.  Furthermore, if $H$ is a subgroup of $G$ such that the restriction $\varphi|_H$ to $H$ is irreducible, and $[G^\ast:C_{G^\ast}(t)]_{p'}>\mathfrak{m}_\ell(H)$, then $\varphi$ cannot be a member of $\mathcal{E}_\ell(G, (t))$.  Also, any irreducible Brauer character in $\mathcal{E}_\ell(G, (t))$ appears as a constituent of the restriction to $\ell$-regular elements for some ordinary character in $\mathcal{E}(G, (t))$ (see \cite[Theorem 3.1]{hissregssblocks}), so $\mathcal{E}_\ell(G, (1))$ is a union of unipotent blocks.  In particular, if $\varphi|_H$ is irreducible and  $[G^\ast:C_{G^\ast}(t)]_{p'}>\mathfrak{m}_\ell(H)$ for all nonidentity semisimple $\ell'$- elements $t$ of $G^\ast$, then $\varphi$ must belong to a unipotent block.

In \cite{bonnaferouquier}, Bonnaf{\'e} and Rouquier show that when $C_{\underline{G}^\ast}(t)$ is contained in an $F^\ast$-stable Levi subgroup, $\underline{L}^\ast$, of $\underline{G}^\ast$, then Deligne-Lusztig induction $R_L^G$ yields a Morita equivalence between $\mathcal{E}_\ell(G,(t))$ and $\mathcal{E}_\ell(L,(t))$, where $L=(\underline{L})^{F}$ and $(\underline{L}, F)$ is dual to $(\underline{L}^\ast, F^\ast)$.  This fact will be very important in what follows.

Note that when $G=Sp_6(q)$, $q$ even, with $G=\underline{G}^F$ and $(\underline{G}^\ast, F^\ast)$ in duality with $(\underline{G}, F)$, each semisimple conjugacy class $(s)$ of $G^\ast=(\underline{G}^\ast)^{F^\ast}$ satisfies that $|s|$ is odd.  Hence by  \cite[Lemma 13.14(iii)]{dignemichel}, the centralizer $C_{\underline{G}^\ast}(s)$ is connected.

\begin{lemma}\label{lem:moritaequiv}
Let $G^\ast=Sp_6(q)$, $q$ even, with $G=\underline{G}^F$ and $(\underline{G}^\ast, F^\ast)$ in duality with $(\underline{G}, F)$.  The nontrivial semisimple conjugacy classes $(s)$ of $G^\ast$ each satisfy $C_{\underline{G}^\ast}(s)=\underline{L}^\ast$ for an $F^\ast$-stable Levi subgroup $\underline{L}^\ast$ of $\underline{G}^\ast$ with $C_{G^\ast}(s)=(\underline{L}^\ast)^{F^\ast}=:L^\ast$.  In particular, Bonnaf{\'e}-Rouquier's theorem \cite{bonnaferouquier} implies that there is a Morita equivalence $\mathcal{E}_\ell(G,(t))\leftrightarrow\mathcal{E}_\ell(L,(1))$ given by Deligne-Lusztig induction when $t\neq 1$ is a semisimple $\ell'$-element, where $L=(\underline{L})^{F}$ and $(\underline{L}, F)$ is dual to $(\underline{L}^\ast, F^\ast)$.
\end{lemma}
\begin{proof}
Write $G^\ast=(\underline{G}^\ast)^{F^\ast}$, as above.  Direct calculation shows that for each semisimple element $s\neq 1$ of $G^\ast$, $C_{\underline{G}^\ast}(s)\leq C_{\underline{G}^\ast}(S)$ for some $F^\ast$-stable torus $S$ in $\underline{G}^\ast$ containing $s$.  (Each such $s$ is conjugate in $\underline{G}^\ast$ to a diagonal matrix $s'=gsg^{-1}$, $g\in\underline{G}^\ast$, whose centralizer in $\underline{G}^\ast$ depends only on the number of distinct entries different than $1$ and their multiplicities.  Hence we may choose $S$ to be $g^{-1}S'g$, where $S'$ is the torus consisting of all diagonal matrices in $\underline{G}^\ast$ with the same form as $s'$.)  Therefore, $C_{\underline{G}^\ast}(s)= C_{\underline{G}^\ast}(S)$, which is an $F^\ast$-stable Levi subgroup of $\underline{G}^\ast$.

Let $t$ be a semisimple $\ell'$-element of $G^\ast$.  Writing $\underline{L}^\ast=C_{\underline{G}^\ast}(t)$, we see that $t\in Z(\underline{L}^\ast)$ and therefore $t\in Z(L^\ast)$.  But then by \cite[Proposition 13.30]{dignemichel},  tensoring with a suitable linear character yields a Morita equivalence of $\mathcal{E}_\ell(L,(t))\leftrightarrow\mathcal{E}_\ell(L,(1))$.  Hence there is a Morita equivalence $\mathcal{E}_\ell(G,(t))\leftrightarrow\mathcal{E}_\ell(L,(t))\leftrightarrow\mathcal{E}_\ell(L,(1))$ by this fact and Bonnaf{\'e}-Rouquier's theorem \cite{bonnaferouquier}.
\end{proof}

\begin{proposition}\label{prop:jordandecompforibrs}
In the notation of \prettyref{lem:moritaequiv}, let $t$ be a semisimple $\ell'$-element of $G^\ast$.  Let $\theta\in\mathcal{E}_\ell(G,(t))$ be an irreducible Brauer character.  Then $\theta(1)=[G^\ast:C_{G^\ast}(t)]_{2'}\varphi(1)$ for some $\varphi\in \ibr_\ell(L)$ lying in a unipotent block of $L$.
\end{proposition}
\begin{proof}

From \prettyref{lem:moritaequiv}, Deligne-Lusztig induction $R_L^G$ provides a Morita equivalence between $\mathcal{E}_\ell(L,(1))$ and $\mathcal{E}_\ell(G,(t))$.  Hence $R_L^G$ gives a bijection between ordinary characters in $\mathcal{E}_\ell(L,(1))$ and $\mathcal{E}_\ell(G,(t))$ and also a bijection between $\ell$-Brauer characters in these two unions of blocks, which preserve the decomposition matrices for these two unions of blocks.

Let $B$ be a unipotent block in $L$, and let $\varphi_1,...,\varphi_m$ be the irreducible Brauer characters in $B$.  Let $\chi_1,...,\chi_s$ be the irreducible ordinary characters in $B$.  Then we can write $\wh{\chi}_i=\sum_{j=1}^m d_{ij}\varphi_j$, where $(d_{ij})$ is the decomposition matrix of the block $B$.  Writing $\psi^\ast$ for the image of an ordinary or Brauer character, $\psi$, of $L$ under Deligne-Lusztig induction $R_L^G$, we therefore also have $\wh{\chi}_i^\ast=\sum_{j=1}^m d_{ij}\varphi_j^\ast$.

Moreover, we may write $\varphi_k=\sum_{i=1}^s a_{ki}\wh{\chi}_i$ for some integers $a_{ki}$.  We claim that $\varphi_k^\ast=\sum_{i=1}^s a_{ki}\wh{\chi}_i^\ast$ as well.  Indeed,
\[\varphi_k=\sum_{i=1}^s a_{ki}\wh{\chi}_i=\sum_{i=1}^s a_{ki}\left(\sum_{j=1}^m d_{ij}\varphi_j\right)=\sum_{j=1}^m \varphi_j\left(\sum_{i=1}^s a_{ki}d_{ij}\right),\]
so $\sum_{i=1}^s a_{ki}d_{ij}=\delta_{kj}$ is the Kronecker delta by the linear independence of irreducible Brauer characters.  Now, \[\sum_{i=1}^s a_{ki}\wh{\chi}_i^\ast=\sum_{i=1}^s a_{ki}\left(\sum_{j=1}^m d_{ij}\varphi_j^\ast\right)=\sum_{j=1}^m\varphi_j^\ast\left(\sum_{i=1}^s a_{ki}d_{ij}\right)=\sum_{j=1}^m\varphi_j^\ast\delta_{kj}=\varphi_k^\ast,\] proving the claim.

Note that $\chi_i^\ast(1)=[G:L]_{2'}\chi_i(1)$ for $1\leq i\leq s$.  Letting $\theta=\varphi_k^\ast$, we can write $\theta=\sum_{i=1}^s a_{ki}\wh{\chi}_i^\ast$, and hence $\theta(1)=\sum_{i=1}^s a_{ki}\wh{\chi}_i^\ast(1)=[G:L]_{2'}\sum_{i=1}^s a_{ki}\wh{\chi}_i(1)=[G:L]_{2'}\varphi_k(1)=[G^\ast:C_{G^\ast}(t)]_{2'}\varphi_k(1)$, which completes the proof.

\end{proof}

While applying Deligne-Lusztig theory to $Sp_{2n}(q)$ with $q$ even, it is convenient to view $Sp_{2n}(q)$ as $SO_{2n+1}(q)\cong Sp_{2n}(q)$, so that $G^\ast = Sp_{2n}(q)$.

\begin{lemma}\label{lem:indexessemisimplecentralizersSp6}
Let $q\geq 4$ be even and let $s\in G^\ast=Sp_6(q)$ be a noncentral semisimple element.  Then either $[G^\ast:C_{G^\ast}(s)]_{2'}\geq (q-1)^2(q^2+1)(q^4+q^2+1),$ or $s$ is a member of one of the classes in \prettyref{tab:ssclasses}, which follows the notation of \cite{Luebeckthesis} and lists the classes in increasing order of $[G^\ast:C_{G^\ast}(s)]_{2'}$.  The table also lists the isomorphism class of $C_{G^\ast}(s)$.
\begin{table}
\caption{Semisimple Classes of $G^\ast=Sp_6(q)$ with Small $[G^\ast:C_{G^\ast}(s)]_{2'}$}
\label{tab:ssclasses}
\centering
\begin{tabular}{|c|c|c|}
  \hline
  Semisimple Class $(s)$& $[G^\ast:C_{G^\ast}(s)]_{2'}$ & $C_{G^\ast}(s)$\\
  \hline
  $c_{4,0}$ & $\frac{q^6-1}{q+1}$ & $Sp_4(q)\times GU_1(q)$\\
  \hline
  $c_{3,0}$ & $\frac{q^6-1}{q-1}$ & $Sp_4(q)\times GL_1(q)$\\
  \hline
  $c_{6,0}$ & $(q^2+1)(q-1)^2(q^2+q+1)$ & $GU_3(q)$\\
  \hline
  $c_{5,0}$ &  $(q^2+1)(q+1)^2(q^2-q+1)$ & $GL_3(q)$\\
  \hline
  $c_{10,0}$ & $(q-1)(q^2+1)(q^4+q^2+1)$ & $GU_2(q)\times Sp_2(q)$\\
  \hline
  $c_{8,0}$ & $(q+1)(q^2+1)(q^4+q^2+1)$ & $GL_2(q)\times Sp_2(q)$\\
\hline
\end{tabular}
\end{table}
\end{lemma}
\begin{proof}
This is evident from inspection of the list of semisimple classes and the sizes of their centralizers in \cite[Tabelles 10 and 14]{Luebeckthesis}.
\end{proof}

\subsection{Other Notes on $Sp_6(q)$, $q$ even}

We note that $|Sp_6(q)|=q^9(q^2-1)(q^4-1)(q^6-1)$, so if $\ell$ is a prime dividing $|Sp_6(q)|$ and $\ell\neq 3$, then $\ell$ must divide exactly one of $q-1$, $q+1$, $q^2+1$, $q^2+q+1$, or $q^2-q+1$.  If $\ell=3$, then it divides $q-1$ if and only if it divides $q^2+q+1$, and it divides $q+1$ if and only if it divides $q^2-q+1$.  In what follows, it will often be convenient to distinguish between these cases.

D. White \cite{white2000} has calculated the decomposition numbers for the unipotent blocks of $Sp_6(q)$, $q$ even, up to a few unknowns in the case $\ell|(q+1)$.  For the convenience of the reader, we summarize these results in \prettyref{app:BrauerCharTable} by describing the $\ell$-Brauer characters for $Sp_6(q)$, $q$ even, lying in unipotent blocks. We give these descriptions in terms of the restrictions of ordinary characters to $\ell$-regular elements.

%% file: PAPERVERSIONlowdimrepsSp6.tex
The purpose of this section is to prove \prettyref{thm:lowdimreps}.  We begin by introducing the Weil characters of $Sp_{2n}(q)$.

\subsection{Weil Characters of  $Sp_{2n}(q)$}

It is convenient to view $Sp_{2n}(q)$ as a subgroup of both $GL_{2n}(q)$ and $GU_{2n}(q)$.  In \cite{TiepGuralnick04}, Tiep and Guralnick describe the linear-Weil characters and unitary-Weil characters, which are irreducible characters of $Sp_{2n}(q)$ for $q$ even and $n\geq 2$ obtained by restriction from $GL_{2n}(q)$ and $GU_{2n}(q)$.   For the convenience of the reader, we recreate \cite[Table 1]{TiepGuralnick04} in \prettyref{tab:TiepGuralnickTable1}.
\begin{table}
\caption{Weil Characters of $Sp_{2n}(q)$ \cite[Table 1]{TiepGuralnick04}}
\label{tab:TiepGuralnickTable1}
\centering
\begin{tabular}{|c|c|c|}
  \hline
  Complex Linear &  & $\ell$-Modular Linear   \\
   Weil Characters & Degree & Weil Characters ($\ell\neq 2$)\\
  \hline
 $ \rho_n^1$ & $\frac{(q^n+1)(q^n-q)}{2(q-1)}$ & $\widehat{\rho}_n^1 - \left\{\begin{array}{cc}
                                                                             1, & \ell|\frac{q^n-1}{q-1},\\
                                                                             0, & \hbox{otherwise}
                                                                           \end{array}\right.$
  \\
  $\rho_n^2$ & $\frac{(q^n-1)(q^n+q)}{2(q-1)}$ & $\widehat{\rho}_n^2 - \left\{\begin{array}{cc}
                                                                             1, & \ell|(q^n+1),\\
                                                                             0, & \hbox{otherwise}
                                                                           \end{array}\right.$\\
  $\tau_n^i,$ & $\frac{q^{2n}-1}{q-1}$ & $\widehat{\tau}_n^i$ \\
  $1 \leq i \leq (q-2)/2$ & & $1\leq i\leq ((q-1)_{\ell'}-1)/2$\\
  \hline
  Complex Unitary &  & $\ell$-Modular Unitary   \\
   Weil Characters & Degree & Weil Characters ($\ell\neq 2$)\\\hline
  $\alpha_n$ & $\frac{(q^n-1)(q^n-q)}{2(q+1)}$ &$ \widehat{\alpha}_n$ \\
  $\beta_n$ & $\frac{(q^n+1)(q^n+q)}{2(q+1)} $ & $\widehat{\beta}_n- \left\{\begin{array}{cc}
                                                                             1, & \ell|(q+1),\\
                                                                             0, & \hbox{otherwise}
                                                                           \end{array}\right.$\\
  $\zeta_n^i,$ & $\frac{q^{2n}-1}{q+1}$ & $\widehat{\zeta}_n^i$,  \\
  $1\leq i\leq q/2$ & & $1\leq i\leq ((q+1)_{\ell'}-1)/2$ \\
  \hline
\end{tabular}
\end{table}

The formulas from \cite{TiepGuralnick04} for calculating the values for the characters $\tau_n^i$ and $\zeta_n^i$ in $SL_{2n}(q)$ and $SU_{2n}(q)$, respectively, are
\begin{equation}\label{formulafortau}
\tau_n^i(g)=\frac{1}{q-1}\sum_{j=0}^{q-2}\tilde{\delta}^{ij}q^{\dim_{\F_q}\ker(g-\delta^j)}-2\delta_{i,0}
\end{equation}
and
\begin{equation}\label{formulaforzeta}
\zeta_n^i(g)=\frac{1}{q+1}\sum_{j=0}^q \tilde{\xi}^{ij}(-q)^{\dim_{\F_{q^2}}\ker(g-\xi^j)}.
\end{equation}
Here $\delta$ and $\tilde{\delta}$ are fixed primitive $(q-1)$th roots of unity in $\F_q$ and $\C$, respectively.  Similarly, $\xi$, $\tilde{\xi}$ are fixed primitive $(q+1)$th roots of unity in $\F_{q^2}$ and $\C$, respectively.  The kernels in the formulae are computed on the natural modules $W:=(\F_q)^{2n}$ for $SL_{2n}(q)$ or $\tilde{W}:=(\F_{q^2})^{2n}$ for $SU_{2n}(q)$.

\subsection{The Proof of \prettyref{thm:lowdimreps}}
 We are now ready to prove \prettyref{thm:lowdimreps}.  We do this in the form of two separate proofs - one for part (A) and one for part (B).

\begin{proof}[Proof of \prettyref{thm:lowdimreps} (A)]
Suppose that $\chi\in\ibr_\ell(G)$ lies in a unipotent block.  The degrees of irreducible Brauer characters lying in unipotent blocks can be extracted from \cite{white2000}, and we have listed these in \prettyref{app:BrauerCharTable}.  Note that the character $\chi_2$ in the notation of \cite{white2000} is the Weil character $\rho_3^2$ in the notation of \cite{TiepGuralnick04}.  Similarly, $\chi_3=\beta_3$, $\chi_4=\rho_3^1$, and $\chi_5=\alpha_3$.

We consider the cases $\ell$ divides $q-1, q+1, q^2-q+1, q^2+q+1, $ and $q^2+1$ separately. Let $D_\ell$ denote the bound in part A(4) of \prettyref{thm:lowdimreps} for the prime $\ell$.

First, assume that $\ell|(q-1)$ and $\ell\neq 3$.  If $\chi(1)\leq D_\ell=\wh{\chi}_{11}(1)$, then since $q\geq 4$, $\chi$ must be $\wh{\chi}_1=1_G, \wh{\chi}_2, \wh{\chi}_3, \wh{\chi}_4, \wh{\chi}_5, \wh{\chi}_6,$ or $\wh{\chi}_7$.  Hence we are in situation A(1), A(2), or A(3).

Now let $\ell|(q^2+q+1)$.  Note that we are including the case $\ell=3|(q-1)$.  In either case, if $\chi(1)\leq D_\ell= \wh{\chi}_{11}(1)$, then $\chi$ is $1_G, \wh{\chi}_2, \wh{\chi}_3, \wh{\chi}_4-1_G, \wh{\chi}_5, \wh{\chi}_6,$ or $\wh{\chi}_7$, as $q\geq 4$.  Again, we therefore have situation A(1), A(2), or A(3).

If $\ell|(q^2+1)$, then again $D_\ell=\wh{\chi}_{11}(1)$.  A character in a unipotent block has degree smaller than this bound if and only if it is $1_G, \wh{\chi}_2, \wh{\chi}_3, \wh{\chi}_4, \wh{\chi}_5, \wh{\chi}_6-1_G,$ or $\wh{\chi}_7-\wh{\chi}_4$, which gives us situation A(1), A(2), or A(3) in this case.

Now let $\ell|(q^2-q+1)$ with $\ell\neq 3$.   Then $D_\ell=\wh{\chi}_{11}(1)-\wh{\chi}_{5}(1)$, and $\chi(1)<D_\ell$ if and only if $\chi$ is $1_G, \wh{\chi}_2-1_G, \wh{\chi}_3, \wh{\chi}_4, \wh{\chi}_5, \wh{\chi}_6$ or $\wh{\chi}_7$, so we have situation A(1), A(2), or A(3) for this choice of $\ell$.

Finally, suppose $\ell|(q+1)$.  In this case, $D_\ell=\varphi_{7}(1)$.  Note that from \cite{white2000}, the parameter $\alpha$ in the description in \prettyref{app:BrauerCharTable} for this Brauer character is $1$ if $(q+1)_{\ell}=3$ and $2$ otherwise.  Also, note that in this case, D. White \cite{white2000} has left $3$ unknowns in the decomposition matrix for the principal block.  Namely, the unknown $\beta_1$ is either $0$ or $1$ and the unknowns $\beta_2, \beta_3$ satisfy
\[1\leq\beta_2\leq(q+2)/2, \quad 1\leq\beta_3\leq q/2.\]  Now, using these bounds for $\beta_2$ and $\beta_3$, we may find a lower bound for $\varphi_{10}(1)$ as follows:
\[\varphi_{10}(1)=\chi_{30}(1)-\beta_3(\chi_{11}(1)-\chi_5(1))-(\beta_2-1)\chi_{23}(1)-\chi_{28}(1)\]
\[=\phi_1^2\phi_3(q^3\phi_4-\beta_3q^4/2+ \beta_3q/2-\phi_4-(\beta_2-1)q\phi_1\phi_6/2) \]
\[\geq \phi_1^2\phi_3(q^3\phi_4-(q/2)q^4/2+ q/2-\phi_4-(q/2)q\phi_1\phi_6/2) = \phi_1^2\phi_3(q^3\phi_4-q^5/4+ q/2-\phi_4-q^2\phi_1\phi_6/4).\]  Here $\phi_j$ represents the $j$th cyclotomic polynomial.  As this bound is larger than $D_\ell$ for $q\geq4$, and the other Brauer characters are known, with the possible exception of $\varphi_2=\wh{\chi}_2-\beta_1\cdot 1_G$, we see that the only irreducible Brauer characters in a unipotent block with degree less than $D_\ell$ are $1_G, \wh{\chi}_2-\beta_1\cdot 1_G, \wh{\chi}_3-1_G, \wh{\chi}_4, \wh{\chi}_5, \wh{\chi}_6-\wh{\chi}_3-\wh{\chi}_2+1_G = \wh{\chi}_{28},$ and $\wh{\chi}_7-\wh{\chi}_6+\wh{\chi}_3-1_G=\wh{\chi}_{35}-\wh{\chi}_5$.

Now, recall that when $\ell|(q^3+1)$, \cite[Table 1]{TiepGuralnick04} gives us that $\wh{\rho}_3^2-1_G$ is an irreducible Brauer character.  Since $(q+1)|(q^3+1)$ and  $\wh{\rho}_3^2=\wh{\chi}_2$, this implies that in fact the unknown $\beta_1$ is $1$.

Hence, we see that we are in one of the situations A(1), A(2), or A(3), and the proof is complete for $\chi$ in a unipotent block.
\end{proof}

\begin{proof}[Proof of \prettyref{thm:lowdimreps}(B)]
As $\chi$ does not lie in a unipotent block, we have $\chi\in\mathcal{E}_\ell(G,(s))$ for some semisimple $\ell'$-element $s\neq 1$. Let $B$ denote the bound $q(q^4+q^2+1)(q-1)^3/2$ in part B(4) of \prettyref{thm:lowdimreps}.  Since $(q-1)^2(q^2+1)(q^4+q^2+1)>B$ when $q\geq 4$, it follows from \prettyref{lem:indexessemisimplecentralizersSp6} and \prettyref{prop:jordandecompforibrs} that either $\chi(1)>B$ or $\chi\in\mathcal{E}_\ell(G,(s))$ where $s$ is lies in one of the classes $c_{3,0}, c_{4,0}, c_{5,0},$ $c_{6,0}, c_{8,0}$, or $c_{10,0}$ of $G^\ast=Sp_6(q)$.  (Note that we are making the identification $G\cong SO_{7}(q)$ so that $G^\ast=Sp_6(q)$ here.)  From \prettyref{tab:ssclasses}, we see that in each of these cases, $C_{G^\ast}(s)=L^\ast$ is a direct product of groups of the form $Sp_2(q), Sp_4(q), GU_i(q), $ or $GL_i(q)$ for $1\leq i\leq 3$, and hence is self-dual.  That is, $L\cong L^\ast$ in the notation of \prettyref{lem:moritaequiv}.  We will make this identification and consider characters of $C_{G^\ast}(s)$ as characters of $L$.

If $s\in c_{3,0}$ or $c_{4,0}$, then $C_{G^\ast}(s)\cong C\times Sp_4(q),$ where $C$ is a cyclic group of order $q-1$ or $q+1$, respectively.  In this case, since $\mathfrak{d}_\ell(Sp_4(q))=(q-1)(q^2-q)/2$ (see \cite{landazuriseitz}), we have $\chi(1)\geq (q^6-1)(q-1)(q^2-q)/(2(q+1))= B$ by \prettyref{prop:jordandecompforibrs},  unless $\chi$ corresponds to $1_{C_{G^\ast}(s)}$ in $\ibr_\ell(C_{G^\ast}(s))$.  In the latter case, we are in situation B(1), as $\chi$ is one of the characters $\wh{\tau}_3^i$ or $\wh{\zeta}_3^j$.

For $s$ in one of the families of classes $c_{5,0}$ or $c_{6,0}$, we have $C_{G^\ast}(s)\cong GL_3(q)$ or $GU_3(q)$, respectively.  Now, nonprincipal characters found in a unipotent $\ell$-block of $GL_3(q)$ have degree at least $q^2+q-1$ (see \cite{james1990}).  Moreover, $\mathfrak{d}_\ell(GU_3(q))$ is at least $q^2-q$ (see, for example, \cite{TiepZalesskii96}). Hence in either case, for $\chi\in\mathcal{E}_\ell(G,(s))$, we know by \prettyref{prop:jordandecompforibrs} that either $\chi(1)\geq(q^2+1)(q-1)^2(q^2+q+1)(q^2-q)>B$ or $\chi$ corresponds to $1_{C_{G^\ast}(s)}$ in $\ibr_\ell(C_{G^\ast}(s))$.  In the second case, we have situation B(2).

Next, suppose that $\chi\in\mathcal{E}_\ell(G,s)$ with $s\in c_{8,0}$ or $c_{10,0}$.  Here we have $C_{G^\ast}(s)\cong GL_2(q)\times SL_2(q)$ or $GU_2(q)\times SL_2(q)$, respectively.  The smallest possible nontrivial character degree in a unipotent block is therefore at least $q-1$.  Since $(q-1)[G^\ast:C_{G^\ast}(s)]_{2'}>B$ in either case, we know by \prettyref{prop:jordandecompforibrs} that either $\chi(1)\geq B$ or situation B(3) holds, and the proof is complete.

\end{proof} 

%% file: PAPERVERSIONCharsRestSp6evencharBasicreduction.tex
The goal of this section is to eliminate many possibilities for subgroups $H$ yielding triples as in \prettyref{prob:aschbakerscottprogram}.  We do this in the form of two theorems, treating $Sp_6(q)$ and $Sp_4(q)$ separately.

\begin{theorem}[Reduction Theorem for $Sp_6(q)$]
 Let $(G,H,V)$ be a triple as in \prettyref{prob:aschbakerscottprogram}, with $\ell\neq 2$, $G=Sp_6(q)$, $q\geq 4$ even, and $H<G$ a maximal subgroup.  Then $H$ is $G$-conjugate to either $G_2(q), O_6^\pm(q),$ or a maximal parabolic subgroup of $G$.
\end{theorem}
\begin{proof}
First note that from \cite{landazuriseitz}, $\mathfrak{d}_\ell(G)= (q^3-1)(q^3-q)/(2(q+1))$. Second, by \cite{roneyholtbray} and \cite{KleidmanLiebeck}, the maximal subgroups of $G$ are isomorphic to one of the following:
\begin{enumerate}
\item $SL_2(q^3).3$
\item $Sp_2(q)\wr S_3$
\item $Sp_4(q)\times Sp_2(q)$
\item $Sp_6(q_0)$, where $q=q_0^m$, some $m>1$
\item $G_2(q)$
\item $O_6^\pm(q)$
\item a maximal parabolic subgroup of $G$.
\end{enumerate}

If $H$ is as in (1), then by Clifford theory, $\mathfrak{m}(H)\leq 3(q^3+1)<\mathfrak{d}_\ell(G)$, since $\mathfrak{m}(SL_2(q^3))=q^3+1$.
If $H$ is as in (2), then $(Sp_2(q))^3\lhd H$ of index $6$, so by Clifford theory, $\mathfrak{m}(H)\leq 6(q+1)^3$, which is smaller than $\mathfrak{d}_\ell(G)$ unless $q=4$.  When $q=4$, we can restrict our attention to the Weil characters, by \prettyref{thm:lowdimreps}.  Hence it suffices by \prettyref{lem:Lemma1forO6} and \prettyref{tab:TiepGuralnickTable1} to note that neither $\chi(1)$ nor $\chi(1)-1$ divides $|H|$ for any complex Weil character $\chi$.

If $H$ is as in (3), then $\mathfrak{m}(H)\leq (q^2+1)(q+1)^3$, since by \cite{luebeckwebsite}, $\mathfrak{m}(Sp_2(q))\leq q+1$ and $\mathfrak{m}(Sp_4(q))\leq (q+1)^2(q^2+1)$.  Hence $\mathfrak{m}(H)\leq D$, where $D$ is the bound in part (B) of \prettyref{thm:lowdimreps}, so by \prettyref{thm:lowdimreps}, $\chi$ must either lift to an ordinary character or belong to a unipotent block of $G$.

Moreover, part (A) of \prettyref{thm:lowdimreps} yields that the only irreducible Brauer characters in a unipotent block that do not lift and have degree at most $\mathfrak{m}(H)$ are $\wh{\rho_3^2}-1, \wh{\beta}_3-1$ in the case $\ell|(q+1)$, $\wh{\rho_3^2}-1$ in the case $\ell|(q^2-q+1)$, or $\wh{\rho}_3^1-1$ in the case $\ell|(q^2+q+1)$.  From \cite{luebeckwebsite}, we see that none of the degrees corresponding to these characters occur in $\irr(H)=\irr(Sp_4(q))\otimes \irr(Sp_2(q))$, and moreover none of the degrees of characters in $\irr(G)$ can occur in $\irr(H)$.  Thus by \prettyref{lem:Lemma1forO6}, there are no possible such modules $V$ for this choice of $H$.

Finally, suppose $H$ is as in (4).  Then
$\mathfrak{m}(H)=
\left\{\begin{array}{cc}(q_0^2+1)(q_0^4+q_0^2+1)(q_0+1)^3& q_0> 4\\
                q_0^2(q_0+1)(q_0^2+1)(q_0^4+q_0^2+1) & q_0\leq4\end{array}\right.$
by \cite{luebeckwebsite}, and $\mathfrak{d}_\ell(G)=(q_0^{3m}-1)(q_0^{3m}-q_0^m)/(2(q_0^m+1))$.  Thus
\[\mathfrak{d}_\ell(G)\geq \frac{(q_0^{6}-1)(q_0^{6}-q_0^2)}{2(q_0^2+1)}=\frac{1}{2}q_0^2(q_0^4+q_0^2+1)(q_0^2-1)^2>\mathfrak{m}(H)\] as long as $q_0\geq4$, and we have only to consider the case $H=Sp_6(2)$.  Here as long as $q\geq 8$, we also have $\mathfrak{d}_\ell(G)>\mathfrak{m}(H)$, so we are reduced to the case $H=Sp_6(2), G=Sp_6(4)$.  Then $\mathfrak{m}(H)=512$ and $\mathfrak{d}_\ell(G)=378$.  Moreover, from \prettyref{thm:lowdimreps}, the only irreducible $\ell$-Brauer characters of $G$ which have degree less than or equal to $\mathfrak{m}(H)$ are Weil characters, which are all of the form $\wh{\chi}$ or $\wh{\chi}-1$ for $\chi\in\irr(G)$. Now, from GAP's character table library (see \cite{GAP4}, \cite{GAPctlib}), it is clear that the only $\ell$-Brauer character of $G$ whose degree occurs as a degree of $H$ is $\wh{\alpha_3}$, which has degree  $378$.  However, observing the character values on involutory classes of both $G$ and $H$, we see that $V$ cannot afford this character.  Thus there are no possible triples $(G,V,H)$ with this $G,H$, by \prettyref{lem:Lemma1forO6}.

Therefore, we are left only with subgroups $H$ as in (5)-(7), as claimed.

\end{proof}

\begin{theorem}[Reduction Theorem for $Sp_4(q)$]
 Let $(G,H,V)$ be a triple as in \prettyref{prob:aschbakerscottprogram}, with $\ell\neq 2$, $G=Sp_4(q)$, $q\geq 4$ even, and $H<G$ a maximal subgroup.  Then $H$ is a maximal parabolic subgroup of $G$.
\end{theorem}
\begin{proof}
Let $V$ afford the character $\chi\in\ibr_\ell(G)$.  From \cite{landazuriseitz}, $\mathfrak{d}_\ell(G)=q(q-1)^2/2$, and by \cite{flesner} and \cite{roneyholtbray}, the maximal subgroups of $G$ are
\begin{enumerate}
\item a maximal parabolic subgroup of $G$ (geometrically, the stabilizer of a point or a line)
\item $Sp_2(q)\wr S_2$ (geometrically, the stabilizer of a pair of polar hyperbolic lines)
\item $O_4^\epsilon(q)$, $\epsilon = +$ or $-$
\item $Sp_2(q^2):2$
\item $[q^4]:C_{q-1}^2$
\item $Sp_4(q_0)$, where $q=q_0^m$, some $m>1$
\item $C_{q-1}^2:D_8$
\item $C_{q+1}^2:D_8$
\item $C_{q^2+1}:4$
\item $Sz(q)$ (when $q=2^{m}$ with $m\geq 3$ odd)
\end{enumerate}

If $H$ is as in (2), (3), or (4), then $\mathfrak{m}(H)\leq 2(q+1)^2$ or $2(q^2+1)$, which are smaller than $\mathfrak{d}_\ell(G)$ for $q\geq 8$.  Letting $q=4$, the only members of $\ibr_\ell(G)$ with sufficiently small degree are the $\ell$-modular Weil characters corresponding to $\alpha_2,\beta_2,\rho_2^1,$ and $\rho_2^2$, and hence either lift to an ordinary character or are of the form $\wh{\chi}-1_G$ for an ordinary character $\chi$ of $G$.  Direct calculation using GAP and the GAP character table library (\cite{GAP4}, \cite{GAPctlib}) show that no ordinary character $\chi\in\irr(G)$ satisfies $\chi|_H\in\irr(H)$ or $\chi|_H-1\in\irr(H)$ when $H\cong SL_2(16):2\cong O_4^-(4)$, so by \prettyref{lem:Lemma1forO6}, $H$ cannot be this group.  If $H=O_4^+(4)\cong (SL_2(4)\times SL_2(4)).2$ or $SL_2(4)\wr S_2$, then let $K\lhd H$ denote the subgroup $SL_2(4)\times SL_2(4)$.  By Clifford theory, $\chi|_K$ must either be irreducible or the sum of two irreducible characters of $K$ of the same degree.   By observing the character values of the $\ell$-modular Weil characters listed above and those of $K$ with the proper degree, it is clear that none of these restrict to $K$ in such a way, except possibly $\wh{\alpha_2}$.  Moreover, both $SL_2(4)\wr S_2$ and $O_4^+(4)$ have a unique ordinary character of degree $18$, but observing the values of this character, we see that this is not $\alpha_2|_H$.  Hence by \prettyref{lem:ifacharacterlifts}, $H$ cannot be as in (2), (3), or (4).

If $H$ is as in (5), then it is solvable and by the Fong-Swan theorem, every $\ell$-Brauer character lifts to an ordinary character.  $H$ has a normal subgroup of the form $[q^4]:C_{q-1}$ with quotient group $C_{q-1}$, so by Clifford theory any irreducible character of $H$ has degree $t\cdot\theta(1)$, where $t$ divides $q-1$ and $\theta\in\irr([q^4]:C_{q-1})$.  Since $[q^4]$ is a normal abelian subgroup of $[q^4]:C_{q-1}$, Ito's theorem implies that $\theta(1)$ divides $q-1$.  It follows that any character of $H$ must have degree dividing $(q-1)^2$, which is smaller than $\mathfrak{d}_\ell(G)$, so $H$ cannot be as in (5).

Finally, for $H$ is as in (6),(7), (8), (9), or (10), $\mathfrak{m}_\ell(H)< \mathfrak{d}_\ell(G)$, which leaves (1) as the only possibility for $H$, as stated.

\end{proof} 

%% file: PAPERVERSIONCharsRestSp6G2evenchar.tex
Let $q$ be a power of $2$.  The purpose of this section is to prove part (2) of \prettyref{thm:mainresult}.  Viewing $G_2(q)$ as a subgroup of $Sp_6(q)$, we solve \prettyref{prob:aschbakerscottprogram} for the case $G=Sp_6(q)$, $H=G_2(q)$, and $V$ is a cross-characteristic $G$-module.  That is, we completely classify all irreducible $\ell$-Brauer characters of $Sp_6(q)$, which restrict irreducibly to $G_2(q)$ when $\ell\neq 2$.

For the classes and complex characters of $Sp_6(q)$, we use as reference Frank L{\"u}beck's thesis (see \cite{Luebeckthesis}), in which he finds the conjugacy classes and irreducible complex characters of $Sp_6(q)$.  For $G_2(q)$, we refer to \cite{EnomotoYamada}, in which Enomoto and Yamada find the conjugacy classes and irreducible complex characters of $G_2(q)$.  We adapt the notation of \cite{EnomotoYamada} that $\epsilon\in\{\pm 1\}$ is such that $q\equiv \epsilon (\mod 3)$.

For the $\ell$-Brauer characters of $Sp_6(q)$, we refer to the work done by D. White in \cite{white2000} (see also \prettyref{app:BrauerCharTable}), and for those of  $G_2(q)$ we refer to work by G. Hiss and J. Shamash in \cite{Hiss89}, \cite{HissShamash90}, \cite{shamash87}, \cite{shamash89}, and \cite{shamash92}.  Since many of these references utilize different notations for the same characters, we include a conversion between notations in \prettyref{app:notationschars}.

The first step is to find the fusion of conjugacy classes from $G_2(q)$ into $Sp_6(q)$.

 \subsubsection{Fusion of Conjugacy Classes in $G_2(q)$ into $Sp_6(q)$}\label{sec:fusionG2Sp6even}
 In this section, we compute the fusion of conjugacy classes from $H=G_2(q)$ into $G=Sp_6(q)$.  \prettyref{tab:FusionG2Sp6qeven} summarizes the results.

 \begin{table}[ht]
\centering
 \caption{The Fusion of Classes from $G_2(q)$ into $Sp_6(q)$}
 \label{tab:FusionG2Sp6qeven}
\subfloat[][]{
\begin{tabular}{|c|c|}
 \hline
 Class in  & Class in \\
 $G_2(q)$ &  $Sp_6(q)$\\
 \hline\hline
 $A_0$ & $c_{1,0}$\\
 \hline
 $A_1$ & $c_{1,2}$\\
 \hline
 $A_2$ & $c_{1,4}$\\
 \hline
  $A_{31}$ & $\left\{\begin{array}{cc} c_{1,5}&\hbox{ if $\epsilon=1$,}\\ c_{1,6} & \hbox{if $\epsilon=-1$}\end{array}\right.$\\
  \hline
  $A_{32}$ & $\left\{\begin{array}{cc} c_{1,6}&\hbox{ if $\epsilon=1$,}\\ c_{1,5} & \hbox{if $\epsilon=-1$}\end{array}\right.$\\
  \hline
  $A_4$ & $\left\{\begin{array}{cc} c_{1,5}&\hbox{ if $\epsilon=1$,}\\ c_{1,6} & \hbox{if $\epsilon=-1$}\end{array}\right.$\\
  \hline
  $A_{51}$ & $c_{1,10}$\\
  \hline
  $A_{52}$ & $c_{1,11}$\\
  \hline

\end{tabular}}
\qquad
\subfloat[][]{
\begin{tabular}{|c|c|}
\hline
 Class in  & Class in \\
 $G_2(q)$ &  $Sp_6(q)$\\
 \hline\hline
$B_0$ & $\left\{\begin{array}{cc} c_{5,0}&\hbox{ if $\epsilon=1$,}\\ c_{6,0} & \hbox{if $\epsilon=-1$}\end{array}\right.$\\
  \hline
  $B_1$ & $\left\{\begin{array}{cc} c_{5,1}&\hbox{ if $\epsilon=1$,}\\ c_{6,1} & \hbox{if $\epsilon=-1$}\end{array}\right.$\\
  \hline
  $B_2(0)$ & $\left\{\begin{array}{cc} c_{5,2}&\hbox{ if $\epsilon=1$,}\\ c_{6,2} & \hbox{if $\epsilon=-1$}\end{array}\right.$\\
  \hline
  $B_2(1)$ &  $\left\{\begin{array}{cc} c_{5,2}&\hbox{ if $\epsilon=1$,}\\ c_{6,2} & \hbox{if $\epsilon=-1$}\end{array}\right.$\\
  \hline
  $B_2(2)$ &  $\left\{\begin{array}{cc} c_{5,2}&\hbox{ if $\epsilon=1$,}\\ c_{6,2} & \hbox{if $\epsilon=-1$}\end{array}\right.$\\
  \hline
\end{tabular}}
\qquad
\subfloat[][]{
\begin{tabular}{|c|c|}
\hline
 Class in  & Class in \\
 $G_2(q)$ &  $Sp_6(q)$\\
 \hline\hline
  $C_{11}(i)$ & $c_{14,0}$\\
  \hline
  $C_{12}(i)$ & $c_{14,1}$\\
  \hline
  $C_{21}(i)$ & $c_{8,0}$\\
  \hline
  $C_{22}(i)$ & $c_{8,3}$\\
  \hline
  $C(i,j)$ & $c_{22,0}$\\
  \hline
  $D_{11}(i)$ & $c_{21,0}$ \\
  \hline
  $D_{12}(i)$ & $c_{21,1}$ \\
  \hline
  $D_{21}(i)$ & $c_{10,0}$\\
  \hline
  $D_{22}(i)$ & $c_{10,3}$\\
  \hline
  $D(i,j)$ & $c_{29,0}$\\
  \hline
  $E_1(i)$ & $c_{26,0}$\\
  \hline
  $E_2(i)$ & $c_{24,0}$\\
  \hline
  $E_3(i)$ & $c_{28,0}$\\
  \hline
  $E_4(i)$ & $c_{31,0}$\\
 \hline
\end{tabular}}

\end{table}

 We begin with the unipotent classes.  In the notation of \cite{EnomotoYamada} and \cite{Luebeckthesis}, the unipotent classes of $H$ and $G$, respectively, are:

\begin{center}
 \begin{tabular}{|c|c|c|c|c|c|c|c|c|}
   \hline
   Class in $G_2(q)$ & $A_0$ & $A_1$ & $A_2$ & $A_{31}$ & $A_{32}$ & $A_4$ & $A_{51}$ & $A_{52}$ \\
   \hline
   Order & 1 & 2 & 2 & 4 & 4 & 4 & 8 & 8 \\
   \hline
 \end{tabular}
\end{center}

\begin{center}
 \begin{tabular}{|c|c|c|c|c|c|c|c|c|c|c|c|c|}
   \hline
   Class in $Sp_6(q)$ & $c_{1,0}$ & $c_{1,1}$ & $c_{1,2}$ & $c_{1,3}$ &  $c_{1,4}$ & $c_{1,5}$ & $c_{1,6}$ & $c_{1,7}$ & $c_{1,8}$ & $c_{1,9}$ & $c_{1,10}$ & $c_{1,11}$  \\
   \hline
   Order & 1 & 2 & 2 & 2 & 2 & 4 & 4 & 4 & 4 &4 & 8 & 8  \\
   \hline
 \end{tabular}
\end{center}

Explicit calculations shows that for any element $u\in H$ of order $8$, $u^4$ lies in the class $A_1$.  Similarly, any $u\in G$ of order $8$ satisfies that $u^4$ lies in the class $c_{1,2}$.  Thus the class $A_1$ of $H$ must lie in the class $c_{1,2}$ of $G$.

Now, \cite[Proposition 7.6]{TiepKleshchev10} implies that the characters $\tau_3^i$ for $1\leq i\leq (q-2)/2$ restrict irreducibly from $GL_6(q)$ to the character $\chi_3(i)$ in $G_2(q)$ (in the notation of \cite{EnomotoYamada}).  Using equation \eqref{formulafortau}, \cite[Sections 1 and 4]{Luebeckthesis}, and \cite{EnomotoYamada}, we see that $c_{1,4}$ is the only conjugacy class of $G$ of involutions on which $\tau_3^i$ has the same value, $q^2+q+1$, as on the class $A_2$ in $H$.  This tells us that the class $A_2$ of $H$ must lie in the class $c_{1,4}$ of $G$.

Moreover, $\chi_3(i)=\tau_3^i|_H$ has the value $q+1$ on all classes of order-$4$ elements in $H$.  Among the classes of order-$4$ elements of $G$, $\tau_3^i$ only has this value on the classes $c_{1,5}$ and $c_{1,6}$.  Hence $A_{31}, A_{32},$ and $A_4$ must sit inside $(c_{1,5}\cup c_{1,6})$.  By comparing the orders of centralizers and noting that $|C_H(x)|$ must divide $|C_G(x)|$ for $x\in H$, we deduce that
\[A_{31}, A_4 \subset H\cap \left\{\begin{array}{cc} c_{1,5}&\hbox{ if $\epsilon=1$,}\\ c_{1,6} & \hbox{if $\epsilon=-1$}\end{array}\right. .\]

We claim that the class $A_{32}$ does not fuse with the classes $A_{31}$ and $A_4$ in $Sp_6(q)$.  Indeed, suppose otherwise, so that $A_{31}, A_{32}, A_4$ are all in $\left\{\begin{array}{cc} c_{1,5}&\hbox{ if $\epsilon=1$,}\\ c_{1,6} & \hbox{if $\epsilon=-1$}\end{array}\right.$.  Consider the character $\chi=\chi_{1,2}\in\irr(G)$ in the notation of \cite{Luebeckthesis}.  Note that this character has the same absolute value on all elements of order $8$, namely $\frac{q}{2}$.  Using the fusion of the Borel subgroup $B=UT$ into the parabolic subgroup $P$ of $H$ and the fusion of $P$ into $H$ found in \cite[Tables I-1, II-1]{EnomotoYamada}, together with the fusion of the elements of order $2$ and $4$ from $H$ into $G$ which we know (or are assuming), we calculate that $[\chi_U,\chi_U]$ is not an integer, a contradiction.  Therefore, $A_{32}$ must not fuse with $A_{31}$ and $A_4$,  so
\[A_{32} \subset H\cap \left\{\begin{array}{cc} c_{1,6}&\hbox{ if $\epsilon=1$,}\\ c_{1,5} & \hbox{if $\epsilon=-1$}\end{array}\right. .\]

We return to the remaining unipotent classes (namely, those with elements of order $8$) after calculating the fusion of the non-unipotent classes.

Recall that $W$ and $\tilde{W}$ denote the natural modules for $SL_{6}(q)$ and $SU_6(q)$, respectively.  The eigenvalues of the semisimple elements acting on $W$ or $\tilde{W}$ are clear from the notation for the element in \cite{Luebeckthesis} and \cite{EnomotoYamada}, and comparing the eigenvalues for representatives in $H$ and in $G$ yields the results for the semisimple classes, which can be found in \prettyref{tab:FusionG2Sp6qeven}.

Now, for arbitrary elements, we use the fact that conjugate elements must have conjugate semisimple and unipotent parts.  In the cases of the classes $c_{14,1}(i)$, $c_{21,1}(i)$ in $Sp_6(q)$, these are the only non-semisimple classes with semisimple part in the appropriate class, from which we deduce
\[C_{12}(i)\subset c_{14,1}(i)\cap H\quad\hbox{and}\quad D_{12}(i)\subset c_{21,1}(i)\cap H.\]

For $\mathcal{C}=C_{22}(i),D_{22}(i),B_1$ in $G_2(q)$, comparing the dimensions of the eigenspaces of the unipotent parts of the classes in $Sp_6(q)$ that have semisimple part in the same class as that of the representative for $\mathcal{C}$, we obtain only one possibility in each case, yielding
\[C_{22}(i)\subset c_{8,3}(i),\quad D_{22}(i)\subset c_{10,3}(i),\quad B_1\subset\left\{\begin{array}{cc} c_{5,1}& \hbox{ if $\epsilon=1$}\\ c_{6,1} & \hbox{ if $\epsilon=-1$}\end{array}\right.\]

This leaves only the classes $B_2(0), B_2(1), B_2(2)$, and the classes of elements of order $8$ in $G_2(q)$.  For these classes, we again utilize the fact that the scalar product of characters must be integral.  Note that the character $\rho_3^1$ is the character $\chi_{1,4}$ in the notation of \cite{Luebeckthesis} and the character $\alpha_3$ is the character $\chi_{1,5}$ in the notation of \cite{Luebeckthesis}, and that for the classes whose fusions have been calculated so far, these characters agree with the characters $\theta_2$ and $\theta_2'$ of $G_2(q)$, respectively, in the notation of \cite{EnomotoYamada}.  Also note that to compute $\left[\rho_3^1|_{G_2(q)}, \rho_3^1|_{G_2(q)}\right]$ or $\left[\alpha_3|_{G_2(q)}, \alpha_3|_{G_2(q)}\right]$, the fusion of the order-$8$ classes is not needed, since the absolute value of each of these characters is the same on all such elements of $Sp_6(q)$.

Suppose that any of $B_2(0), B_2(1),$ or $B_2(2)$ fuses with $B_1$ in $Sp_6(q)$.  Then for $\epsilon=1$, $\left[\rho_3^1|_{G_2(q)}, \rho_3^1|_{G_2(q)}\right]$ is not an integer since $[\theta_2,\theta_2]$ is an integer.  If $\epsilon=-1$, then $\left[\alpha_3|_{G_2(q)}, \alpha_3|_{G_2(q)}\right]$ is not an integer, using the fact that $[\theta_2', \theta_2']$ is an integer.  Since there is only one other non-semisimple conjugacy class in $Sp_6(q)$ with the same semisimple part, this contradiction yields that $B_2(0), B_2(1),$ and $B_2(2)$ must fuse in $Sp_6(q)$, and
\[B_2(0)\cup B_2(1)\cup B_2(2)\subset \left\{\begin{array}{cc} c_{5,2}&\hbox{ if $\epsilon=1$,}\\ c_{6,2} & \hbox{if $\epsilon=-1$}\end{array}\right.\cap G_2(q)\]

Finally, we may return to the order-$8$ unipotent classes.  If the two classes $A_{51}, A_{52}$ fused in $Sp_6(q)$, then we would have that $\rho_3^1$ agrees with the character $\theta_2$ on all conjugacy classes of $G_2(q)$ except either $A_{51}$ or $A_{52}$. Using this fact, we can calculate $\left[\rho_3^1|_{G_2(q)}, \theta_2\right]$ to see that it is not an integer, so these two classes cannot fuse.  If $A_{51}$ was contained in $c_{1,11}$ and $A_{52}$ was in $c_{1,10}$, we would again see that $\left[\rho_3^1|_{G_2(q)}, \theta_2\right]$ is not an integer, so we must have
\[A_{51}\subset c_{1,10}\cap G_2(q)\quad\hbox{ and }\quad A_{52}\subset c_{1,11}\cap G_2(q),\] which completes the calculation of the fusions of classes of $G_2(q)$ into $Sp_6(q)$.

\subsubsection{The Complex Case}\label{sec:cxcaseG2toSp6even}
\input{PAPERVERSIONComplexCharsRestSp6G2evenchar.tex}
\subsubsection{The Modular Case}
\input{PAPERVERSIONModularCharsRestSp6G2evenchar.tex}
\subsubsection{Descent to Subgroups of $G_2(q)$}
We now consider subgroups $H$ of $Sp_6(q)$ such that $H<G_2(q)$.  In \cite{Nguyen08}, Nguyen finds all triples as in \prettyref{prob:aschbakerscottprogram} when $G=G_2(q)$ and $H$ is a maximal subgroup.  Noting that none of the representations described in \cite{Nguyen08} to give triples for $G=G_2(q)$ come from the Weil characters listed in \prettyref{thm:resultforG2}, it follows that there are no proper subgroups of $H$ of $G_2(q)$ that yield triples as in \prettyref{prob:aschbakerscottprogram} for $G=Sp_6(q)$. 

%% file: PAPERVERSIONComplexCharsRestSp6G2evenchar.tex
In this section, we consider ordinary characters $\chi\in\irr(Sp_6(q))$ which restrict irreducibly to $G_2(q)$.  We also discuss decomposition of the Weil characters that are reducible over $G_2(q)$.

\begin{theorem}\label{thm:resultcomplexforG2}
Let $G=Sp_6(q)$, $H=G_2(q)$ with $q\geq 4$ even.  Suppose that $V$ is an absolutely irreducible ordinary $G$-module.  Then $V$ is irreducible over $H$ if and only if $V$ affords one of the Weil characters \begin{itemize}
    \item ${\rho}_3^1 $, of degree $\frac{1}{2}q(q+1)(q^3+1)$,
     \item${\tau}_3^i$, $1\leq i\leq ((q-1)_{\ell'}-1)/2$, of degree $(q^2+q+1)(q^3+1)$,
    \item $ {\alpha}_3$, of degree $\frac{1}{2}q(q-1)(q^3-1)$,
    \item ${\zeta}_3^i$, $1\leq i\leq ((q+1)_{\ell'}-1)/2$, of degree $(q^2-q+1)(q^3-1)$.
\end{itemize}
\end{theorem}
\begin{proof}

Assume $V|_H$ is irreducible.  Using \cite{luebeckwebsite} to compare character degrees of $H$ and $G$, we see that the Weil characters $\rho_3^1, \tau_3^i,\alpha_3,\zeta_3^i$ are the only possibilities for the character afforded by $V$.  Thus it suffices to show that each such character is indeed irreducible when restricted to $H$.

Note that from \cite{TiepKleshchev10}, the characters $\tau_3^i$ for $1\leq i\leq (q-2)/2$ actually restrict irreducibly from $GL_6(q)$ to $G_2(q)$, and $\tau_3^i|_{G_2(q)}=\chi_3(i)$ in the notation of \cite{EnomotoYamada}.

We use the fusion of the classes of $H$ into $G$ found in \prettyref{sec:fusionG2Sp6even} to compute the character values of $\zeta_3^i$ on each class.  The class representatives for $G$ found in \cite{Luebeckthesis} are given in their Jordan-Chevelley decompositions, from which we can find the eigenvalues and the dimensions of the eigenspaces over $\F_{q^2}$.  Using the formula \eqref{formulaforzeta}, we then conclude that $\zeta_3^i|_H$ agrees with the character $\chi_3'(i)$ of $H$ in the notation of \cite{EnomotoYamada}, and therefore is irreducible on $H$ for each $1\leq i\leq q/2$.

In the notation of \cite{Luebeckthesis}, $\rho_3^1$ is the unipotent character $\chi_{1,4}$ and $\alpha_3$ is the unipotent character $\chi_{1,5}$.  Given the fusion of classes found in \prettyref{sec:fusionG2Sp6even}, we see that $\chi_{1,4}|_H$ agrees with the character $\theta_2$ in \cite{EnomotoYamada} and $\chi_{1,5}|_H$ agrees with the character $\theta_2'$ in \cite{EnomotoYamada}, meaning that $\rho_3^1$ and $\alpha_3$ are therefore irreducible when restricted to $G_2(q)$.

\end{proof}

\begin{theorem}\label{thm:decompofrho2andbeta}
Let $q$ be a power of $2$.  Then
\begin{enumerate}
\item the linear Weil character $\rho_3^2$ in $\irr(Sp_6(q))$ decomposes over $G_2(q)$ as
\[(\rho_3^2)|_{G_2(q)}=\theta_1+\theta_4,\]

and

\item the unitary Weil character $\beta_3$ in $\irr(Sp_6(q))$ decomposes over $G_2(q)$ as
\[(\beta_3)|_{G_2(q)}=\theta_1'+\theta_4,\] where $\theta_1,\theta_1',\theta_4\in\irr(G_2(q))$ are the characters of degrees $\frac{1}{6}q(q+1)^2(q^2+q+1),$ $\frac{1}{6}q(q-1)^2(q^2-q+1),$ and $\frac{1}{3}q(q^4+q^2+1)$, respectively, as in the notation of Enomoto and Yamada, \cite{EnomotoYamada}.
\end{enumerate}
\end{theorem}
\begin{proof}
This follows from the fusion of conjugacy classes found in \prettyref{sec:fusionG2Sp6even} and the character tables in \cite{Luebeckthesis} and \cite{EnomotoYamada}, noting that the character $\rho_3^2$ and $\beta_3$ are given by $\chi_{1,2}$ and $\chi_{1,3}$, respectively, in the notation of \cite{Luebeckthesis}.
\end{proof}

%% file: PAPERVERSIONModularCharsRestSp6G2evenchar.tex
In this section, we consider more generally the irreducible Brauer characters $\chi\in\ibr_\ell(Sp_6(q))$ in characteristic $\ell\neq 2$ which restrict irreducibly to $G_2(q)$.

\begin{theorem}\label{thm:charswhichrestrictG2}
Let $G=Sp_6(q)$, $H=G_2(q)$ with $q\geq 4$ even.  Let $\ell\neq 2$ and suppose $\chi\in\ibr_\ell(G)$ is one of the following:
\begin{itemize}
    \item $\widehat{\rho}_3^1 - \left\{\begin{array}{cc}
                                                                             1, & \ell|\frac{q^3-1}{q-1},\\
                                                                             0, & \hbox{otherwise}
                                                                           \end{array}\right.$,
                                                                            \item $\widehat{\tau}_3^i$, $1\leq i\leq ((q-1)_{\ell'}-1)/2$,
                                                                                \item $ \widehat{\alpha}_3$, \item $\widehat{\zeta}_3^i$, $1\leq i\leq ((q+1)_{\ell'}-1)/2$.
\end{itemize}
Then $\chi|_H\in\ibr_\ell(H)$.
\end{theorem}
\begin{proof}
We may assume that $\ell||G|$, since otherwise the result follows from \prettyref{thm:resultcomplexforG2}.  We consider the cases $\ell$ divides $(q-1), (q+1), (q^2-q+1),(q^2+q+1),$ and $(q^2+1)$ separately.

If $\ell|(q-1)$, then $(\rho_3^1)|_{H}=X_{15}$ in \cite{HissShamash90},\cite{Hiss89}.  From \cite[Table I]{HissShamash90}, we see that if $\ell=3$, then indeed $\widehat{X}_{15}-1_H$ is an irreducible Brauer character of $H$.  From \cite{Hiss89}, we see that if $\ell\neq 3$, then $\widehat{X}_{15}$ is an irreducible Brauer character.  We also see that $(\alpha_3)|_H$ has defect $0$, so indeed $(\widehat{\alpha}_3)|_H\in\ibr_\ell(H)$.

By \cite{HissShamash90} and \cite{Hiss89}, $(\widehat{\zeta}_3^i)_H=\widehat{X}'_{2a}$ is an irreducible Brauer character, and the $((q-1)_{\ell'}-1)/2$ characters $(\widehat{\tau}_3^i)|_H=\widehat{X}'_{1b}$ which lie outside the the principal block are also irreducible Brauer characters, completing the proof in the case $\ell|(q-1)$.

Now let $\ell|(q+1)$.  In this case, Hiss and Shamash show in \cite{HissShamash90} and \cite{Hiss89} that $(\widehat{\tau}_3^i)|_H=\widehat{X}'_{1b}$ is an irreducible Brauer character and the $((q+1)_{\ell'}-1)/2$ characters  $(\widehat{\zeta}_3^i)_H=\widehat{X}'_{2a}$ lying outside the principal block are irreducible Brauer characters.  Also, from \cite[Section 3.3]{HissShamash90} and \cite[Section 2.2]{Hiss89}, $\widehat{X}_{17}=\widehat{\alpha}_3|_H\in\ibr(H)$.  Finally, note that $(\rho_3^1)|_H$ has defect 0, which completes the proof in the case $\ell|(q+1)$.

Suppose $\ell|(q^2-q+1)$, where $\ell\neq 3$.  From \cite[Section 2.1]{shamash92}, we see that $X_{17}$ lies in the principal block with cyclic defect group and that $\widehat{X}_{17}\in\ibr(H)$.  As this character is the restriction of $\alpha_3$ to $H$, we have $(\widehat{\alpha}_3)|_H\in\ibr(H)$.  We see from their degrees that $X_{15}, X_{1b}'$, and $X_{2a}'$ are all of defect $0$, so their restrictions to $\ell$-regular elements are irreducible Brauer characters of $H$.  But these are exactly the restrictions to $H$ of the characters $\rho_3^1, \tau_3^i,$ and $\zeta_3^i$, respectively, which completes the proof in the case $\ell|(q^2-q+1)$.

Now assume $\ell|(q^2+q+1)$, where $\ell\neq 3$.  Then from the Brauer tree for $H$ given in \cite[Section 2.1]{shamash92}, we see that $\widehat{X}_{15}-1\in\ibr(H)$, and since $(\rho_3^1)|_H=X_{15}$ in Shamash's notation, this shows that $\widehat{\rho_3^1}-1$ restricts irreducibly to $H$.  Also, $X_{17}, X_{2a}',$ and $X_{1b}'$ have defect 0, so $\widehat{X}_{17},\widehat{X}_{2a}',$ and $\widehat{X}_{1b}'\in\ibr(H)$ as well.  As $(\alpha_3)|_H=X_{17}, (\zeta_3^k)|_H=X_{2a}',$ and $(\tau_3^k)|_H=X_{1b}'$ in Shamash's notation, it follows that all of the characters claimed indeed restrict irreducibly to $H$, completing the proof in the case $\ell|(q^2+q+1)$.

Finally, if $\ell|(q^2+1)$, then $\ell$ does not divide $|H|$, which means that $\ibr(H)=\irr(H)$, and the result is clear from \prettyref{thm:resultcomplexforG2}.

\end{proof}

\begin{theorem}\label{thm:resultforG2}
Let $G=Sp_6(q)$, $H=G_2(q)$ with $q\geq 4$ even.  Suppose that $V$ is an absolutely irreducible $G$-module in characteristic $\ell\neq 2$.  Then $V$ is irreducible over $H$ if an only if the $\ell$-Brauer character afforded by $V$ is one of the Weil characters \begin{itemize}
    \item $\widehat{\rho}_3^1 - \left\{\begin{array}{cc}
                                                                             1, & \ell|\frac{q^3-1}{q-1},\\
                                                                             0, & \hbox{otherwise}
                                                                           \end{array}\right.$,
                                                                            \item $\widehat{\tau}_3^i$, $1\leq i\leq ((q-1)_{\ell'}-1)/2$,
                                                                                \item $ \widehat{\alpha}_3$, \item $\widehat{\zeta}_3^i$, $1\leq i\leq ((q+1)_{\ell'}-1)/2$.
\end{itemize}
\end{theorem}
\begin{proof}
 If $V$ affords one of the characters listed, then $V$ is irreducible on $H$ by \prettyref{thm:charswhichrestrictG2}.  Conversely, assume that $V$ is irreducible on $H$ and let $\chi\in\ibr_\ell(G)$ denote the $\ell$-Brauer character afforded by $V$.  If $\chi$ lifts to a complex character, then the result follows from \prettyref{thm:resultcomplexforG2}, so we assume $\chi$ does not lift.   We may therefore assume that $\ell$ is an odd prime dividing $|G|$.   We note that $\chi(1)\leq \mathfrak{m}(H)\leq(q+1)^2(q^4+q^2+1)$ by \cite{luebeckwebsite}, and if $q=4$, then $\mathfrak{m}(H)=q(q+1)(q^4+q^2+1)$.

Since $(q-1)(q^2+1)(q^4+q^2+1)>\mathfrak{m}(H)$ when $q\geq 4$, it follows from part (B) of \prettyref{thm:lowdimreps} that either $\chi$ lifts to an ordinary character or $\chi$ lies in a unipotent block of $G$.  In the first situation, \prettyref{thm:resultcomplexforG2} implies that $\chi$ is in fact one of the characters listed in the statement.  Therefore, we may assume that $\chi$ lies in a unipotent block of $G$ and does not lift to a complex character.

Since $\mathfrak{m}(H)$ is smaller than the degree of each of the characters listed in situation A(3) of \prettyref{thm:lowdimreps}, we see that the only irreducible Brauer characters which do not lift to a complex character and whose degree does not exceed $\mathfrak{m}(H)$ are $\wh{\rho}_3^2-1_G$ and $\wh{\beta}_3-1_G$ when $\ell|(q+1)$, $\wh{\rho}_3^2-1_G$ in the case $3\neq\ell|(q^2-q+1)$, $\wh{\rho}_3^1-1_G$ in the case $\ell|(q^2+q+1)$, and $\wh{\chi}_6-1_G$ when $\ell|(q^2+1)$.

From \prettyref{thm:decompofrho2andbeta}, we know that $(\rho_3^2)|_{G_2(q)}=\theta_1+\theta_4$ and $(\beta_3)|_{G_2(q)}=\theta_1'+\theta_4$ in the notation of \cite{EnomotoYamada}.  Also, $\theta_4=X_{14}, \theta_1=X_{16},$ and $\theta_1'=X_{18}$ in the notation of Shamash and Hiss.

Suppose $\ell|(q+1)$.  From \cite[Section 2.2]{Hiss89}, we know that $\widehat{X}_{14}-1\in\ibr_\ell(H)$ when $\ell\neq 3$, and therefore neither $\widehat{\rho_3^2}-1$ nor $\widehat{\beta}_3-1$ can restrict irreducibly to $\ibr_\ell(H)$.  If $\ell=3$, then by \cite[Section 3.3]{HissShamash90},  $\widehat{X}_{14}+\widehat{X}_{18}-1\not\in\ibr_\ell(H)$, since this is $\varphi_{14}+2\varphi_{18}$ in the notation of \cite[Table II]{HissShamash90}.  Similarly, $\widehat{X}_{14}+\widehat{X}_{16}-1\not\in\ibr_\ell(H)$, so we have shown that if $\ell=3$, again neither $\widehat{\rho_3^2}-1$ nor $\widehat{\beta}_3-1$ can restrict irreducibly to $\ibr_\ell(H)$.

Suppose $\ell|(q^2-q+1)$, where $\ell\neq 3$.  From \cite[Section 2.1]{shamash92}, the Brauer character $\widehat{X}_{16}-1$ is irreducible, and $X_{14}$ is defect zero, so $\widehat{X}_{14}$ is also irreducible.  But this means that $\widehat{X}_{14}+\widehat{X}_{16}-1$ is not irreducible.  Recalling again that $X_{14}=\theta_4$ and $X_{16}=\theta_1$, this shows that $\widehat{\rho}_3^2-1_G$ does not restrict irreducibly to $H$.

If $\ell|(q^2+q+1)$, then we are done by \prettyref{thm:charswhichrestrictG2}.  Finally, if $\ell|(q^2+1)$, then $\ell$ cannot divide $|H|$, which means that $\ibr_\ell(H)=\irr(H)$, and every irreducible Brauer character of $H$ lifts to $\C$.  Since the degree of $\widehat{\chi}_6-\widehat{\chi}_1$ is not the degree of any element of $\irr(H)$, we know $\chi$ cannot be $\widehat{\chi}_6-\widehat{\chi}_1$, and the proof is complete.

\end{proof} 

%% file: PAPERVERSIONCharsRestSp6O6evenchar.tex
In this section, let $q\geq 4$ be a power of $2$, $G=Sp_6(q)$, and $H^\pm\cong O_6^\pm(q)$ as a subgroup of $G$.  Since $q$ is even, we have $H^\pm=\Omega_6^\pm(q).2\cong L_4^\pm(q).2$ (see \cite[Chapter 2]{KleidmanLiebeck}).  We will denote by $K^\pm$ the index-$2$ subgroup $L_4^\pm(q)$ of $H^\pm$.  We at times may simply refer to $H, K$ rather than $H^\pm, K^\pm$ if the result is true in either case.

The purpose of this section is to show that restrictions of nontrivial representations of $G$ to $H$ are reducible.  We again begin with the complex case.
\begin{theorem}\label{thm:complexresultO6}
Let $G=Sp_6(q)$ and $H=O^\pm_6(q)$, with $q\geq 4$ even.  If $1_G\neq\chi\in\irr(G)$, then $\chi_H$ is reducible.
\end{theorem}
\begin{proof}
Assume that $\chi|_H$ is irreducible.  For the list of irreducible complex character degrees of $K^\pm\cong L_4^\pm(q)$ and $G=Sp_6(q)$, we refer to \cite{luebeckwebsite}.  From Clifford theory, $\chi_H$ has degree $e\cdot \phi(1)$ where $e\in\{1,2\}$ and $\phi\in\irr(K^\pm)$.  Inspecting the list of character degrees for $K^\pm$ and for $G$, it follows that for $q>4$, the only option for $\chi(1)$ is $(q^2+1)(q^2-q+1)(q+1)^2$ in case $-$ and $(q^2+1)(q^2+q+1)(q-1)^2$ in case $+$, and that $e=1$.  Hence from \cite{Luebeckthesis}, $\chi$ is $\chi_{8,1}$, or $\chi_{9,1}$, respectively.  However, by inspecting the character values on involutory classes, it is clear that neither of these characters restrict irreducibly to $H^\pm$.  (Here we have used the character tables for $GL_4(q)\cong C_{q-1}\times L^+_4(q)$ and $GU_4(q)\cong C_{q+1}\times L^-_4(q)$ constructed by F. L{\"u}beck for the CHEVIE system \cite{chevie}.)  Therefore, for $q>4$, $\chi|_H$ must be reducible.

In the case $q=4$, there are additional character degrees $\phi(1)$ of $K$ for which $2\phi(1)$ is a character degree for $G$.  These degrees are $221$ and $325$ for $K^-\cong SU_4(4)$, or $189$ and $357$ for $K^+\cong SL_4(4)$.  For each of these degrees, there is exactly one character of $Sp_6(4)$ with twice that degree.

Suppose that $\chi(1)=442$.  Then $\chi=\beta_3$, and using the GAP Character Table Library \cite{GAPctlib} and calculation in GAP, we see that $\beta_3$ restricts to $K^-$ as the sum of the two characters of $K^-$ of degree $221$.  Moreover, calculation in GAP \cite{GAP4} shows that these two characters are fixed by the order-2 automorphism of $K^-$ inside $H^-$ given by $\tau\colon (a_{ij})\mapsto (a_{ij}^q)$, and hence extend to $H^-$.  (Note that $\tau$ is the automorphism of $K^-$ inside $H^-$, since $\mathrm{Out}(K^-)$ is cyclic so has only one order-2 outer automorphism.)  Thus the restriction of $\beta_3$ is reducible.

There are two characters of degree $189$ in $\irr(SL_4(4))$, and one of degree $378$ in $G$ (namely, $\alpha_3$), and from direct calculation in GAP, we see that the restriction of $\alpha_3$ to $K^+$ is the sum of these two characters.   The order - 2 automorphism of $K^+$ inside $H^+$ is given by the graph automorphism $\sigma\colon A\mapsto (A^{-1})^T$.  (Indeed, by \cite[Chapter 2]{KleidmanLiebeck}, the isomorphism of $L_4^+(q)$ with $\Omega_6^+(q)$ is given by the identification of $A\in L_4^+(q)$ with its action on the second wedge space of the natural module, and $\Omega_6^+(q)$ is the index-2 subgroup of $O_6^+(q)$ composed of elements that can be written as a product of an even number of reflections.  Hence it suffices to note that $\sigma$ can be identified with conjugation in $O_6^+(q)$ by a suitable product of an odd number of reflections.)  Again using calculations in GAP, we see that these characters of $K^+$ extend to irreducible characters of $H^+$, since they are fixed by $\sigma$.  Thus the restriction of $\alpha_3$ is reducible.

There is exactly one character, $\phi$, of degree $325$ in $\irr(SU_4(4))$, which means that if $\chi(1)=650$, then $\chi|_{K^-}=2\phi$.  Now, as $H^-/K^-$ is cyclic and $\phi$ is $H^-$-invariant, we see that $\phi$ must extend to a character of $H^-$, so $\chi|_{K^-}\neq 2\phi$.

Similarly, there is exactly one character, $\phi$, of degree $357$ in $\irr(SL_4(4))$, which means that if $\chi(1)=714$, then the restriction of $\chi$ to $K^+$ is twice this character.  Again, as $H^+/K^+$ is cyclic and $\phi$ is $H^+$-invariant, this is not the case.

\end{proof}

\begin{lemma}\label{lem:Lemma2forO6}
Let $q\geq4$ and let $\chi\in\irr(G)$ be one of the characters $\chi_2,\chi_3,\chi_4,\chi_6$ in the notation of \cite{white2000}.  If $\chi|_H-\lambda\in\irr(H)$ for $\lambda\in \wh{H}$, then the restriction to $K$ also satisfies $\chi|_K-\lambda|_K\in\irr(K)$.
\end{lemma}
\begin{proof} Writing $\theta:=\chi|_H-\lambda\in\irr(H)$ and noting $[H:K]=2$, we know by Clifford theory that $\theta_K=\sum_{i=1}^t\theta_i$ where $\theta_i\in\irr(K)$,  each $\theta_i$ has the same degree, and $t|2$.  Since $\theta(1)=\chi(1)-1$ is odd, it follows that $\theta_K$ is irreducible.
\end{proof}
\begin{lemma}\label{lem:Lemma3forO6}
Let $q\geq 4$ and $\chi$ be one of the characters as in \prettyref{lem:Lemma2forO6}.  Then $\chi|_H-\lambda\not\in\irr(H)$ for any $\lambda\in \wh{H}\cup\{0\}$.  In particular, $\wh{\chi}_H-1_H\not\in\ibr_\ell(H)$ for any prime $\ell$.
\end{lemma}
\begin{proof}
Comparing degrees of characters of $G$ and $K$ (see, for example, \cite{luebeckwebsite}), we see that neither $\chi(1)$ nor $\chi(1)/2$ occur as a degree of an irreducible character of $K$ for any of these characters. Then by Clifford theory (see the argument in \prettyref{lem:Lemma2forO6}), we know that $\chi|_H\not\in\irr(H)$.  Moreover, $\chi(1)-1$ does not occur as an irreducible character degree for $K$, which means that $\chi|_K-\lambda_K\not\in\irr(K)$ for any $\lambda\in\wh{H}$.  Thus by \prettyref{lem:Lemma2forO6}, $\chi|_H-\lambda\not\in\irr(H)$ for any $\lambda\in\wh{H}$.  The last statement then follows by \prettyref{lem:Lemma1forO6}.

\end{proof}

We are now ready to prove the following theorem, which generalizes \prettyref{thm:complexresultO6} to the modular case:
\begin{theorem}\label{thm:resultforO6}
Let $H\cong O_6^{\pm}(q)$ be a maximal subgroup of $G=Sp_6(q)$, with $q\geq 4$ even, and let $\ell\neq 2$ be a prime.  If $\chi\in\ibr_\ell(G)$ with $\chi(1)>1$, then the restriction $\chi|_H$ is reducible.
\end{theorem}
\begin{proof}

Suppose that $\chi|_H$ is irreducible.  We first note that from Clifford theory, $\mathfrak{m}_\ell(H^\pm)=\mathfrak{m}_\ell(K^\pm.2)\leq 2\mathfrak{m}_\ell(K^\pm)$.  Now $\mathfrak{m}_\ell(K^+) \leq (q+1)^2(q^2+1)(q^2+q+1)$ and $\mathfrak{m}_\ell(K^-)\leq (q+1)^2(q^2+1)(q^2-q+1)$ (see, for example, \cite{luebeckwebsite}).

Note that $q(q^4+q^2+1)(q-1)^3/2 >\mathfrak{m}_\ell(H^-)$ for $q\geq 4$.  Moreover, $q(q^4+q^2+1)(q-1)^3/2 >\mathfrak{m}_\ell(H^+)$, except possibly when $q=4$.  However, from \cite{luebeckwebsite}, we can see that if $q=4$, then in fact $\mathfrak{m}_\ell(K^+)\leq 7140$, so $q(q^4+q^2+1)(q-1)^3/2 >\mathfrak{m}_\ell(H^+)$ in this case as well.  Thus we know from \prettyref{thm:lowdimreps} that either $\chi$ lifts to a complex character, or $\chi$ lies in a unipotent block.

Suppose that $\chi$ lies in a unipotent block of $G$.  Then the character degrees listed in situation A(3) of \prettyref{thm:lowdimreps} are larger than our bound for $\mathfrak{m}_\ell(H^-)$ for $q\geq 4$ and are larger than $\mathfrak{m}_\ell(H^+)$ unless $q=4$ and $\ell|(q+1)$.  (Here we have again used the fact that $\mathfrak{m}_\ell(K^+)\leq 7140$.)  Hence, by \prettyref{thm:lowdimreps}, $\chi$ either lifts to an ordinary character or is of the form $\wh{\chi}-1_G$ where $\chi$ is one of the characters discussed in \prettyref{lem:Lemma3forO6} (and therefore do not remain irreducible over $H$), except possibly in the case $H=O_6^+(4)$ and $\ell=5$.

If $q=4$ and $\ell=5$, the bound $D$ in part (A) of \prettyref{thm:lowdimreps} is larger than 14280, so $\wh{\chi}_{35}-\wh{\chi}_5$ is the only additional character we must consider.  However, the degree of $\wh{\chi}_{35}-\wh{\chi}_5$ is $(q^3-1)(q^4-q^3+3q^2/2-q/2+1)=13545$, which is odd, so by Clifford theory, if it restricts irreducibly to $H^+$, then it also restricts irreducibly to the index-$2$ subgroup $K^+$.  But $7140<13545$, a contradiction. Hence $\wh{\chi}_{35}-\wh{\chi}_5$ is reducible when restricted to $H^+$.

We have therefore reduced to the case of complex characters, which by \prettyref{thm:complexresultO6} are all reducible on $H$.

\end{proof} 

%% file: PAPERVERSIONCharsRestSp6Parabolicsevenchar.tex
The purpose of this section is to prove part (1) of \prettyref{thm:mainresult}.  We momentarily relax the assumption that $G=Sp_6(q)$, and instead consider the more general case $G=Sp_{2n}(q)$ for $n\geq 2$.  Let $\{e_1,...,e_n,f_1,...,f_n\}$ denote a symplectic basis for the natural module $\F_q^{2n}$.  That is, $(e_i,e_j)=(f_i,f_j)=0$ and $(e_i,f_j)=\delta_{ij}$ for $1\leq i,j\leq n$, so that the gram matrix of the symplectic form with isometry group $G$ is $J_{n}:=\left(
                                  \begin{array}{cc}
                                    0 & I_n \\
                                    I_n & 0 \\
                                  \end{array}
                                \right)$.  We will use many results from \cite{TiepGuralnick04} and will keep the notation used there.  In particular, $P_j=\stab_G(\langle e_1,...,e_j\rangle_{\F_q})$ will denote the $j$th maximal parabolic subgroup, $L_j$ its Levi subgroup, $Q_j$ its unipotent radical, and $Z_j=Z(Q_j)$.

If we reorder the basis as $\{e_1,...,e_n, f_{j+1},...,f_n,f_1,...,f_j\}$, then the subgroup $Q_j$ can be written as
\[Q_j=\left\{\left(
               \begin{array}{ccc}
                 I_j & (A^T)J_{n-j} & C\\
                 0 & I_{2n-2j} & A \\
                 0 & 0 & I_j \\
               \end{array}
             \right)
\colon A\in M_{2n-2j,j}(\F_q), C\in M_j(q), C+C^T+(A^T)J_{n-j}A=0\right\}\] and
\[Z_j=\left\{\left(
               \begin{array}{ccc}
                 I_j & 0 & C\\
                 0 & I_{2n-2j} & 0 \\
                 0 & 0 & I_j \\
               \end{array}
             \right)
\colon  C\in M_j(q), C+C^T=0\right\}.\]  In particular, note that in the case $j=n$, $Q_n$ is abelian and $Z_n=Q_n$.  Also, $L_j\cong Sp_{2n-2j}(q)\times GL_j(q)$ is the subgroup
\[L_j=\left\{\left(
               \begin{array}{ccc}
                 A & 0&0\\
                 0 & B & 0 \\
                 0 & 0 & (A^T)^{-1} \\
               \end{array}
             \right)
\colon A\in GL_{j}(q), B\in Sp_{2n-2j}(q)\right\}.\]

 Linear characters $\lambda\in\irr(Z_j)$ are in the form
\[\lambda_Y\colon \left(
               \begin{array}{ccc}
                 I_j & 0 & C\\
                 0 & I_{2n-2j} & 0 \\
                 0 & 0 & I_j \\
               \end{array}
             \right) \mapsto (-1)^{\tr_{\F_q/\F_2}(\tr(YC))}\] for some $Y\in M_j(q)$.  These characters correspond to quadratic forms $q_Y$ on $\F_q^j=\langle f_1,...,f_j\rangle_{\F_q}$ defined by $q_Y(f_i)=Y_{ii}$ with associated bilinear form having Gram matrix $Y+Y^T$.  The $P_j$-orbit of the linear characters $\lambda_Y$ of $Z_j$ is given by the rank $r$ and type $\pm$ of $q_Y$, denoted by $\mathcal{O}_r^\pm$ for $0\leq r\leq j$.  We will sometimes denote the corresponding orbit sums by $\omega_r^\pm$. For $\lambda\in\mathcal{O}_r^\pm$,
\[\stab_{L_j}(\lambda)\cong Sp_{2n-2j}(q)\times \left([q^{r(j-r)}]\colon (GL_{j-r}(q)\times O_r^\pm(q))\right),\] where $[N]$ denotes the elementary abelian group of order $N$.

We begin with a theorem proved in \cite{tiep06}.

\begin{theorem}\label{thm:thm1.6fromtiep}
Let $G=Sp_{2n}(q)$.  Let $Z$ be a long-root subgroup and assume $V$ is a non-trivial irreducible representation of $G$.  Then $Z$ must have non-zero fixed points on $V$.
\end{theorem}
\begin{proof}
This is \cite[Theorem 1.6]{tiep06} in the case that $G$ is type $C_n$.
\end{proof}

\prettyref{thm:thm1.6fromtiep} shows that there are no examples of irreducible representations of $G$ which are irreducible when restricted to $P_1$.
\begin{corollary}\label{cor:resultforP1}
Let $V$ be an irreducible representation of $G=Sp_{2n}(q)$, $q$ even, which is irreducible on $H=P_1=\stab_G(\langle e_1\rangle_{\F_q})$.  Then $V$ is the trivial representation.
\end{corollary}
\begin{proof}
Suppose that $V$ is non-trivial and let $\chi\in\ibr_\ell(G)$ denote the Brauer character afforded by $V$. By Clifford theory, $\chi|_{Z_1}= e\sum_{\lambda\in\mathcal{O}}\lambda$ for some $P_1$-orbit $\mathcal{O}$ on $\irr(Z_1)$ and positive integer $e$.   But in this case, $Z_1$ is a long-root subgroup, so $Z_1$ has non-zero fixed points on $V$ by \prettyref{thm:thm1.6fromtiep}.  This means that $\mathcal{O}=\{1_{Z_1}\}$, so $Z_1\leq \ker\chi$, a contradiction since $G$ is simple.

\end{proof}

 We can view $Sp_4(q)$ as a subgroup of $G$ under the identification $Sp_4(q)\simeq \stab_{G}(e_3,...,e_n,f_3,...,f_n)$.  To distinguish between subgroups of $Sp_4(q)$ and $Sp_{2n}(q)$, we will write $P_j^{(n)}=\stab_{Sp_{2n}(q)}(\langle e_1,...,e_j\rangle)$ for the $j$th maximal parabolic subgroup of $Sp_{2n}(q)$, $P_j^{(2)}$ for the $j$th maximal parabolic subgroup of $Sp_{4}(q)$, and similarly for the subgroups $Z_j, Q_j,$ and $L_j$.
 Note that $P_2^{(2)}\leq P_n^{(n)}$ and $Z_2^{(2)}\leq Z_n^{(n)}$.

 The following theorem will often be useful when viewing $Sp_4(q)$ as a subgroup of $G$ in this manner.

 \begin{theorem}\label{thm:alpha2irredP2}
  Let $q$ be even and let $V$ be an absolutely irreducible $Sp_4(q)$-module of dimension larger than $1$ in characteristic $\ell\neq 2$.  Then $V$ is irreducible on $P_2=\stab_G(\langle e_1,e_2\rangle_{\F_q})$ if and only if $V$ affords the $\ell$-Brauer character $\wh{\alpha}_2$.
 \end{theorem}
 \begin{proof}
Let $Z:=Z_2^{(2)}$ be the unipotent radical of $P_2$.  First we claim that $\wh{\alpha}_2$ is indeed irreducible on $P_2$.  Note that $\wh{\alpha}_2|_Z=\alpha_2|_Z$ since $Z$ consists of $2$-elements.  Now, $\alpha_2(1)=|\mathcal{O}_2^-|$, and by Clifford theory it suffices to show that $\alpha_2|_{Z}=\sum_{\lambda\in\mathcal{O}_2^-}\lambda = \omega_2^-$.  From the proof of \cite[Proposition 4.1]{TiepGuralnick04}, it follows that nontrivial elements of $Z$ belong to the classes $A_{31}, A_2, A_{32}$ of $Sp_4(q)$.  Inspecting the values of $\alpha_2$ and $\omega_2^-$ on these classes, which are found in the proof of \cite[Proposition 4.1]{TiepGuralnick04} and \cite{enomoto72}, respectively, we see that $\alpha_2|_Z=\omega_2^-$, and $\wh{\alpha}_2$ must be irreducible when restricted to $P_2$.

Conversely, suppose that $\chi$ is the Brauer character afforded by $V$, and $\chi|_{P_2}=\varphi\in\ibr_\ell(P_2)$.  By Clifford theory, $\varphi|_{Z}=e\sum_{\lambda\in\mathcal{O}}\lambda$ for some nontrivial $P_2$-orbit $\mathcal{O}$ of $\irr(Z)$.  It follows that $\varphi$ satisfies condition $\mathcal{W}_2^\pm$ of \cite{TiepGuralnick04}, so $\chi$ is a Weil character of $Sp_4(q)$ by \cite[Theorem 1.2]{TiepGuralnick04}.

Now, following the notation of the proof of \cite[Proposition 4.1]{TiepGuralnick04}, we have  \[\zeta_2|_{Z}=1_{Z}+(q+1)\omega_1+(2q+2)\omega_2^-.\]

Since $Z$ consists of $2$-elements, \cite[Lemma 3.8]{TiepGuralnick04} implies that $\zeta_2^i|_Z=\alpha_2|_Z+\beta_2|_Z-1_Z$, so by the definition of $\zeta_2$ (see \cite[Section 3]{TiepGuralnick04}),
 \[\zeta_2|_{Z}=(q+1)\alpha_2|_{Z}+(q+1)\beta_2|_{Z}-q\cdot 1_{Z}.\]  
 It follows that
 $\wh{\beta}_2|_{Z}=1_{Z}+\omega_1+\omega_2^-.$

The values of $\omega_1$ and $ \omega_2^+$ on $Z$ are obtained in \cite[Proposition 4.1]{TiepGuralnick04}, and the values of $\rho_2^1$, and $\rho_2^2$ are obtained in \cite{enomoto72}. Inspection of these values on the classes $A_{31}, A_2, A_{32}$ yields that $\rho_2^1|_Z=\omega_2^++q\cdot 1_Z$ and $\rho_2^2|_Z= (q+1)\cdot 1_Z+\omega_1+\omega_2^+$.  Moreover, \cite[Lemma 3.8]{TiepGuralnick04} implies that $\tau_2^i|_Z=\rho_2^1|_Z+\rho_2^2|_Z+1$.

Hence, we see that if $\chi$ is any Weil character aside from $\wh{\alpha}_2$, then $\chi|_Z$ contains as constituents multiple $P_2$-orbits of characters of $Z$, a contradiction.
 \end{proof}
 The following corollary follows directly from the proof of \prettyref{thm:alpha2irredP2}.
 \begin{corollary}\label{cor:beta2onZ2}
 Let $Z_2$ be the unipotent radical of $P_2=\stab_{Sp_4(q)}(\langle e_1,e_2\rangle_{\F_q})$.  Then
 \[\alpha_2|_{Z_2} = \sum_{\lambda\in\mathcal{O}_2^-}\lambda, \qquad \beta_2|_{Z_2}=\sum_{\lambda\in\mathcal{O}_2^-}\lambda+\sum_{\lambda\in\mathcal{O}_1}\lambda+1_{Z_2},\qquad\zeta_2^i|_{Z_2}=2\sum_{\lambda\in\mathcal{O}_2^-}\lambda+\sum_{\lambda\in\mathcal{O}_1}\lambda,\]
 \[\rho_2^1|_{Z_2}=q\cdot1_{Z_2}+\sum_{\lambda\in\mathcal{O}_2^+}\lambda,\qquad \rho_2^2|_{Z_2}=(q+1)\cdot 1_{Z_2}+\sum_{\lambda\in\mathcal{O}_1}\lambda+\sum_{\lambda\in\mathcal{O}_2^+}\lambda,\] and \[\tau_2^i|_{Z_2}= (2q+2)\cdot 1_{Z_2}+\sum_{\lambda\in\mathcal{O}_1}\lambda+2\sum_{\lambda\in\mathcal{O}_2^+}\lambda.\]
 \end{corollary}

 \begin{theorem}\label{thm:alphanirredonPn}
 Let $G=Sp_{2n}(q)$ with $q$ even and $n\geq 2$, and let $V$ be an absolutely irreducible $G$-module  in characteristic $\ell\neq 2$ affording the $\ell$-Brauer character $\wh{\alpha}_n$.  Then $V$ is irreducible on $P_n=\stab_G(\langle e_1,...,e_n\rangle_{\F_q})$.
 \end{theorem}
 \begin{proof}
Note that $\ibr_\ell(Z_n)=\irr(Z_n)$ since $Z_n$ is made up entirely of $2$-elements.  Let $\lambda_Y\in\irr(Z_n)$ be labeled by $Y=\left(
                                                \begin{array}{cc}
                                                  Y_1 &  Y_2\\
                                                  Y_3 & Y_4 \\
                                                \end{array}
                                              \right)\in M_n(q)$ with $Y_1\in M_2(q), Y_4\in M_{n-2}(q)$.  Identifying a symmetric matrix $X\in M_2(q)$ with both $\left(
                                                \begin{array}{cc}
                                                  I_2 &  X\\
                                                  0 & I_2\\
                                                \end{array}
                                              \right)\in Z_2^{(2)}$ and $\left(
                                                \begin{array}{cc}
                                                  I_n &  X_1\\
                                                  0 & I_n\\
                                                \end{array}
                                              \right)\in Z_n^{(n)}$, where $X_1:=\left(
                                                \begin{array}{cc}
                                                  X &  0\\
                                                  0 & 0\\
                                                \end{array}
                                              \right)\in M_n(q),$ we see
 \[\lambda_Y(X)=(-1)^{\mathrm{Tr}_{\F_q/\F_2}(\mathrm{Tr}(X_1Y))}=(-1)^{\mathrm{Tr}_{\F_q/\F_2}(\mathrm{Tr}(XY_1))}=\lambda_{Y_1}(X).\] Thus $\lambda_Y|_{Z_2^{(2)}}=\lambda_{Y_1}$.  Also, it is clear from the definition that $q_Y|_{\langle f_1,f_2\rangle_{\F_q}}=q_{Y_1}$.

 From \cite[Proposition 7.2]{TiepGuralnick04}, $\wh{\alpha}_n|_{Sp_{2n-2}(q)}$ contains $\wh{\alpha}_{n-1}$ as a constituent, and continuing inductively, we see $\wh{\alpha}_n|_{Sp_{4}(q)}$ contains $\wh{\alpha}_2$ as a constituent.  Now, by \prettyref{thm:alpha2irredP2}, $\wh{\alpha}_2$ is irreducible when restricted to $P_2^{(2)}$, and $\wh{\alpha_2}|_{Z_2^{(2)}}$ is the sum of the characters in the orbit $\mathcal{O}_2^-$.

 Since $\wh{\alpha}_2|_{Z_2^{(2)}}$ is a constituent of $\wh{\alpha}_n|_{Z_{2}^{(2)}}$, it follows that  $\wh{\alpha}_n|_{Z_{n}^{(n)}}$ must contain some $\lambda_Y$ such that $q_{Y_1}$ is rank-$2$.  Since $|\mathcal{O}_2^-|=\alpha_n(1)$ and $|\mathcal{O}_r^\pm|>\alpha_n(1)$ for the other orbits with $r\geq 2$, we know $\wh{\alpha}_n|_{Z_n^{(n)}}=\sum_{\lambda\in\mathcal{O}_2^-}\lambda$.  Therefore $\wh{\alpha}_n|_{P_n^{(n)}}$ must be irreducible.

 \end{proof}

 It will now be convenient to reorder the basis of $G=Sp_{2n}(q)$ as $\{e_1, e_2,...,e_n,f_3,f_4,...,f_n,f_1,f_2\}$.  Under this basis, the embedding of $Sp_4(q)$ into $G$ is given by
 \[Sp_4(q)\ni\left(
     \begin{array}{cc}
       A & B \\
       C & D \\
     \end{array}
   \right)\mapsto \left(
                    \begin{array}{ccc}
                      A & 0 & B \\
                      0 & I_{2n-4} & 0 \\
                      C & 0 & D \\
                    \end{array}
                  \right)\in Sp_{2n}(q)
 \] where $A,B,C,D$ are each $2\times 2$ matrices.

 Note that $P_2^{(2)}\leq P_2^{(n)}$ and, moreover, $Z_2^{(2)} = Z_2^{(n)}$.  We will therefore simply write $Z_2$ for this group.

 \begin{theorem}\label{thm:generalresultforP2}
Let $G=Sp_{2n}(q)$ with $q$ even and $n\geq 2$, and let $V$ be an absolutely irreducible $G$-module with dimension larger than $1$ in characteristic $\ell\neq 2$.  Then $V$ is absolutely irreducible on $P_2^{(n)}$ if and only if $n=2$ and $V$ is the module affording the $\ell$-Brauer character $\wh{\alpha}_2$.
 \end{theorem}
 \begin{proof}
  Assume $n>2$. Let $\chi\in\ibr_\ell(G)$ denote the $\ell$-Brauer character afforded by $V$, and let $\varphi\in\ibr_\ell(H)$ be the $\ell$-Brauer character afforded by $V$ on $H:=P_2^{(n)}$.  Write $Z:=Z_2$.  The nontrivial orbits of the action of $H$ on $\irr(Z)$ and those of $P_2^{(2)}$ on $\irr(Z)$ are the same, with sizes
 \[|\mathcal{O}_1|=q^2-1,\quad |\mathcal{O}_2^-|=\frac{1}{2}q(q-1)^2,\quad |\mathcal{O}_2^+|=\frac{1}{2}q(q^2-1).\]

 By Clifford theory, $\chi|_Z=e\sum_{\lambda\in\mathcal{O}}\lambda$ for one of these orbits $\mathcal{O}$ and some positive integer $e$.  (Note that $\mathcal{O}$ is not the trivial orbit since $G$ is simple, so $\chi$ cannot contain $Z$ in its kernel.)  It is clear from this that $V|_H$ has the property $\mathcal{W}_2^\pm$  in the notation of \cite{TiepGuralnick04}, and therefore by \cite[Theorem 1.2]{TiepGuralnick04}, $\chi$ is one of the Weil characters from \prettyref{tab:TiepGuralnickTable1}.

 If $\chi$ is a linear Weil character, then the branching rules found in \cite[Propositions 7.7]{TiepGuralnick04} imply that $\chi|_{Sp_4(q)}$ contains $1_{Sp_4(q)}$ as a constituent, and so $\chi|_Z$ contains $1_Z$ as a constituent, which is a contradiction.

 If $\chi$ is a unitary Weil character, then the branching rules found in \cite[Proposition 7.2]{TiepGuralnick04} show that $\chi|_{Sp_4(q)}$ contains as a constituent $\sum_{k=1}^{q/2}\wh{\zeta}_2^k$.   But \cite[Lemma 3.8]{TiepGuralnick04} shows that $\zeta_n^i=\alpha_n+\beta_n-1$ on $Z$, so by \prettyref{cor:beta2onZ2}, $\chi|_Z$ contains $(q/2)(\omega_1+2\omega_2^-)$, a contradiction since $\chi|_Z$ can have as constituents $Z$-characters from only one $H$-orbit.

 We therefore see that $n$ must be $2$, and the result follows from \prettyref{thm:alpha2irredP2}.
 \end{proof}

\begin{corollary}\label{cor:resultforparabolicsSp4}
Let $q$ be even.  A nontrivial absolutely irreducible representation $V$ of $Sp_4(q)$ in characteristic $\ell\neq 2$ is irreducible on a maximal parabolic subgroup if and only if the subgroup is $P_2$ and $V$ affords the character $\wh{\alpha}_2$.
\end{corollary}
\begin{proof}
This is immediate from \prettyref{thm:generalresultforP2} and \prettyref{cor:resultforP1}.
\end{proof}
Note that we have now completed the proof of \prettyref{thm:mainresultSp4}.

We will now return to the specific group $G=Sp_6(q)$.  Let $H=P_3=\stab_G(\langle e_1,e_2,e_3\rangle_{\F_q})$ be the third maximal parabolic subgroup, and note that here $Z_3=Q_3$ is elementary abelian of order $q^6$.  We will simply write $Z$ for this group.  The sizes of the four nontrivial orbits of $\irr(Z)$ and the corresponding $L_3$-stabilizers are
 \[|\mathcal{O}_1|=q^3-1,\quad |\stab_{L_3}(\lambda)|=q^3(q-1)(q^2-1);\]
 \[|\mathcal{O}_2^\pm|=\frac{1}{2}q(q\pm 1)(q^3-1),\quad |\stab_{L_3}(\lambda)|=2q^2(q-1)(q\mp 1);\] and
 \[|\mathcal{O}_3|=q^2(q-1)(q^3-1),\quad |\stab_{L_3}(\lambda)|=q(q^2-1).\]

 We begin by considering the ordinary case, $\ell=0$.
 \begin{theorem}\label{thm:complexcaseP3}
 Let $V$ be a nontrivial absolutely irreducible ordinary representation of $G=Sp_6(q)$, $q\geq 4$ even.  Then $V$ is irreducible on $H=P_3$ if and only if it affords the Weil character $\alpha_3$.
 \end{theorem}
 \begin{proof}
 Note that $\alpha_3$ is irreducible on $H$ by \prettyref{thm:alphanirredonPn}.  Conversely, suppose that $\chi\in\irr(G)$ is irreducible when restricted to $H$.  Since $Z\lhd H$ is abelian, it follows from Ito's theorem that $\chi(1)$ divides $[H:Z]=q^3(q-1)(q^2-1)(q^3-1).$  Moreover, by Clifford theory, if $\lambda\in\irr(Z)$ such that $\chi|_H\in\irr(H|\lambda)$, then $\chi(1)$ is divisible by the size of the $H$-orbit $\mathcal{O}$ containing $\lambda$.  In particular, this means that $q^3-1$ must divide $\chi(1)$.  (Note that $\lambda\neq 1$, since $G$ is simple and thus $Z$ cannot be contained in the kernel of $\chi$.)  However, from inspection of the character degrees given in \cite{luebeckwebsite}, it is clear that the only irreducible ordinary character of $G$ satisfying these conditions is $\alpha_3$.

 \end{proof}

 Given any $\varphi\in\ibr_\ell(H)$ and a nontrivial irreducible constituent $\lambda$ of $\varphi|_Z$, we know by Clifford theory that $\varphi=\psi^H$ for some $\psi\in\ibr_\ell(I|\lambda)$, where $I=\stab_H(\lambda)$.  Then $\psi|_Z=\psi(1)\cdot \lambda$ and therefore $\ker\lambda\in\ker\psi$.  Note that $|Z/\ker\lambda|=2$ since $Z$ is elementary abelian and $\lambda$ is nontrivial.   Viewing $\psi$ as a Brauer character of $I/\ker\psi$, we see
\[\psi(1)\leq \sqrt{|I/\ker\psi|}\leq\sqrt{|I/\ker\lambda|}=\left(\frac{|Z|\cdot|\stab_{L_3}(\lambda)|}{\ker\lambda}\right)^{1/2}=\sqrt{2|\stab_{L_3}(\lambda)|}\]

Now, $\varphi(1)=\psi(1)\cdot|\mathcal{O}|$ where $\mathcal{O}$ is the $H$-orbit of $\irr(Z)$ which contains $\lambda$.  If $\lambda\in\mathcal{O}_1$, this yields
\[\varphi(1)\leq (q^3-1)\sqrt{2q^3(q-1)(q^2-1)}=(q-1)(q^3-1)\sqrt{2q^3(q+1)},\] and we will denote this upper bound by $B_1$.

If $\lambda\in\mathcal{O}_2^\pm$, then we see similarly that
\[\varphi(1)\leq \frac{1}{2}q(q\pm1)(q^3-1)\sqrt{4q^2(q-1)(q\mp1)}.\]  We will denote this bound by $B_2^\pm$, so
\[B_2^-:=q^2(q-1)(q^3-1)\sqrt{q^2-1}, \quad\hbox{ and }\quad B_2^+:=q^2(q^2-1)(q^3-1).\]

For $\lambda\in\mathcal{O}_3$, we have $I=Z\colon Sp_2(q)$.  If we denote $K:=\ker\psi$, then $(K\cdot Sp_2(q))/K\leq I/K$.  But $(K\cdot Sp_2(q))/K\cong Sp_2(q)/(K\cap Sp_2(q))\cong Sp_2(q)$ or $\{1\}$ since $Sp_2(q)$ is simple for $q\geq 4$.  Thus either $I/\ker\psi$ contains a copy of $Sp_2(q)$ as a subgroup of index at most $2$ or $\psi(1)=1$.  Moreover, $(ZK)/K\lhd I/K$. But $(ZK)/K\cong Z/(Z\cap K)=Z/\ker\lambda\cong\Z/2\Z$, and thus $I/K$ contains a normal subgroup of size $2$.  Assuming we are in the case that $I/K$ contains a copy of $Sp_2(q)$, we know this normal subgroup intersects $Sp_2(q)$ trivially, and thus $I/K\cong \Z/2\times Sp_2(q)$.  In either case, $\psi(1)\leq \mathfrak{m}(Sp_2(q))=q+1$, and therefore
\[\varphi(1)\leq (q+1)q^2(q-1)(q^3-1)=q^2(q^2-1)(q^3-1),\] which we will denote by $B_3$.  Note that $B_3=B_2^+>B_2^->B_1$ for $q\geq 4$.

\begin{theorem}\label{thm:resultforP3}
Let $G=Sp_6(q)$, $q\geq 4$ even, and let $H=P_3$.  Then a nontrivial absolutely irreducible $G$-module $V$ in characteristic $\ell\neq 2$ is irreducible on $H$ if and only if $V$ affords the $\ell$-Brauer character $\wh{\alpha}_3$.
\end{theorem}
\begin{proof}
That $\wh{\alpha}_3$ is irreducible on $H$ follows from \prettyref{thm:alphanirredonPn}.  Conversely, suppose that $V$ affords $\chi\in\ibr_\ell(G)$ and that $\chi|_H=\varphi\in\ibr_\ell(H)$.  We claim that $\chi$ must lift to an ordinary character, so the result follows from \prettyref{thm:complexcaseP3}.  We will keep the notation from the above discussion.

First suppose that $\chi$ does not lie in a unipotent block.  As the bound $q(q-1)^3(q^4+q^2+1)/2$ in part (B) of \prettyref{thm:lowdimreps} is larger than $B_2^-$ and is larger than $B_3$ unless $q=4$, it follows that either $\chi$ lifts to an ordinary character or $q=4$ and $\lambda\in\mathcal{O}_3$ or $\mathcal{O}_2^+$.

Now let $q=4$.  We identify $G$ with $SO_{7}(4)$ so that $G^\ast=Sp_6(4)$.  Let $\mathfrak{u}(C_{G^\ast}(s))$ denote the smallest degree larger than $1$ of an irreducible Brauer character lying in a unipotent block of $C_{G^\ast}(s)$ for a semisimple element $s$.  Using the same argument as in the proof of part (B) of \prettyref{thm:lowdimreps}, we note that for a nontrivial semisimple element $s\in G^\ast$, $\mathfrak{u}(C_{G^\ast}(s))[G^\ast:C_{G^\ast}(s)]_{2'}>B_3$ unless $s$ belongs to a class in the family $c_{3,0}$ or $c_{4,0}$.  In this case, $C_{G^\ast}(s)\cong Sp_4(q)\times C$ for a cyclic group $C$.

Now, the Brauer character tables of $Sp_4(4)$ are available in the GAP Character Table Library, \cite{GAP4},\cite{GAPctlib}.  We can see that the smallest nonprincipal character degree of $Sp_4(4)$ for any $\ell\neq 2$ is $18$.  This corresponds to $\wh{\alpha}_2$, which clearly lifts to $\C$, so by the Morita equivalence guaranteed by \prettyref{lem:moritaequiv}, $\chi$ also lifts if it corresponds to this character.  The next smallest degree is $33$ if $\ell=5$ and $34$ if $\ell=3,17$.    If $s\in c_{3,0}$, then $[G^\ast:C_{G^\ast}(s)]_{2'}=1365$, and $1365\cdot33= 45045>15120=B_3$.  If $s\in c_{4,0}$, then $[G^\ast:C_{G^\ast}(s)]_{2'}=819$, and $819\cdot33=27027>15120=B_3$.  It follows that in this case, $\chi$ must again lift to an ordinary character.

Now assume $\chi$ lies in a unipotent block.  Note that the bound $D$ in part (A) of \prettyref{thm:lowdimreps} is larger than $B_3$ for $q\geq 4$.  Hence, $\chi$ must be as in situations A(1), A(2), or A(3) of \prettyref{thm:lowdimreps}.  Also, note that $\chi(1)$ must be divisible by $(q^3-1)$, as $|\mathcal{O}_1|$, $|\mathcal{O}_2^\pm|$, and $|\mathcal{O}_3|$ are all divisible by $(q^3-1)$.  Therefore, $\chi$ cannot be any of the characters $\wh{\rho}_3^1-1, \wh{\rho}_3^2-1,$ $\wh{\beta}_3-1$, $\wh{\chi}_6-1$ or $\wh{\chi}_7-\wh{\chi}_4$.  Thus in the case $\ell|(q^3-1)(q^2+1)$ or $3\neq\ell|(q^2-q+1)$, we know from \prettyref{thm:lowdimreps} that $\chi$ lifts to an ordinary character.

Now assume $\ell|(q+1)$ and that $\chi$ does not lift to an ordinary character.  Then by the above remarks, $\chi$ must be $\wh{\chi}_{35}-\wh{\chi}_5$, which has degree larger than $B_2^-$ and is odd.  Since $|\mathcal{O}_3|$ and $|\mathcal{O}_2^+|$ are each even, this shows our $\chi$ cannot be this character.  So, $\chi$ must again lift to an ordinary character.

\end{proof}

\begin{corollary}\label{cor:resultforparabolics}
Let $G=Sp_6(q)$ with $q\geq 4$ even.  A nontrivial absolutely irreducible $G$-module $V$ in characteristic $\ell\neq 2$ is irreducible on a maximal parabolic subgroup $P$ if and only if $P=P_3$ and $V$ affords the $\ell$-Brauer character $\wh{\alpha}_3$.
\end{corollary}
\begin{proof}
This follows directly from \prettyref{cor:resultforP1}, \prettyref{thm:generalresultforP2}, and \prettyref{thm:resultforP3}.

\end{proof}

\subsection{Descent to Subgroups of $P_3$}

Let $Z=Z_3$ be the unipotent radical of $P_3$, and let $R\leq Z$ be the subgroup $[q^3]$ given by matrices $C\in Z$ with zero diagonal.  That is, \[R=\left\{\left(
           \begin{array}{cc}
             I & C \\
             0 & I \\
           \end{array}
         \right)\colon C\in M_3(q),  C+C^T=0, \hbox{ $C$ has diagonal $0$}\right\}.\]

Note that the subgroups $L_3\cong GL_3(q)$ and $L_3'\cong SL_3(q)$ of $P_3$ each act transitively on $R\setminus 0$.

Let $\lambda=\lambda_Y$ be an irreducible character of $Z$  corresponding to the matrix $Y\in M_3(q)$, and write $\lambda|_R=\mu=\mu_Y$.  If $\lambda'$ is another such character corresponding to $Y'$ and $\lambda'|_R=\mu'$, then we have $\mu=\mu'$ if and only if $(Y+Y')+(Y+Y')^T=0$.  (Note that unlike characters of $Z$, we do not require that $Y, Y'$ have the same diagonal.)  Hence $\mu_Y=\mu_{X^TYX}$ for $X\in GL_3(q)$ if and only if $X^T(Y+Y^T)X=Y+Y^T$.  That is, $X$ is in the isometry group of the form with Gram matrix $Y+Y^T$.  As the action of $X\in L_3\cong GL_3(q)$ on $\mu_Y$ is given by $(\mu_Y)^X=\mu_{X^TYX}$, this means that $\stab_{L_3}(\mu)$ is this isometry group..

In particular, if $\lambda$ is in the $P_3$-orbit $\mathcal{O}_2^-$ of linear characters of $Z$, then this means that $\stab_{L_3}(\lambda|_R)=[q^2]: (\F_q^\times\times Sp_2(q))=[q^2]:GL_2(q)$.  Recall that from the proof of \prettyref{thm:alphanirredonPn}, $\alpha_3|_Z=\omega_2^-$ is the orbit sum corresponding to $\mathcal{O}_2^-$.  Hence we have
\[\stab_{L_3}(\mu)=[q^2]:GL_2(q)\]
if $\mu$ is a constituent of $\alpha_3|_R$.   Taking the elements of this stabilizer with determinant one, we also see
\[\stab_{L_3'}(\mu)=[q^2]:SL_2(q).\]

\begin{lemma}\label{lem:alphairredZSL}
The Brauer character $\wh{\alpha}_3$ is irreducible on the subgroup $P_3'=Z:SL_3(q)$ of $P_3$.
\end{lemma}
\begin{proof}
Let $\lambda$ be an irreducible constituent of $\wh{\alpha}_3|_{Z}$, so that $\lambda\in\mathcal{O}_2^-$.  Recall that the stabilizer in $L_3\cong GL_3(q)$ is $\stab_{L_3}(\lambda)\cong[q^2]:(\F_q^\times \times O_2^-(q))$.  Taking the elements in this group with determinant $1$, we see that the stabilizer in $SL_3(q)$ is isomorphic to $[q^2]:(\mathcal{O}_2^-(q))$, and hence the $P_3'$-orbit has length
\[\frac{q^9(q^2-1)(q^3-1)}{2q^8(q+1)}=\frac{1}{2}q(q-1)(q^3-1)=|\mathcal{O}_2^-|=\alpha_3(1).\] Therefore, $\wh{\alpha}_3|_{P_3'}$ is irreducible.

\end{proof}

\begin{lemma}\label{lem:subgpcontainsSL}
Let $G=Sp_6(q)$ with $q\geq 4$ even, and let $V$ be an absolutely irreducible $G$-module $V$ which affords the Brauer character $\wh{\alpha}_3$.  Write $Z=Z_3$ for the unipotent radical of the parabolic subgroup $P_3$ and $L=L_3$ for the Levi subgroup.  If $H<P_3$ with $V|_H$ irreducible, then $ZH$ contains $P_3'=Z:L'=Z:SL_3(q)$.
\end{lemma}
\begin{proof}
Note that $HZ/Z\cong H/(Z\cap H)$ is a subgroup of $P_3/Z\cong GL_3(q)$.  As $\alpha_3(1)=q(q-1)(q^3-1)/2$, we know that $|H|_{2'}$ is divisible by $(q-1)(q^3-1)$.  Moreover, $HZ/Z$ must act transitively on the $q^3-1$ elements of $R\setminus0$.  Therefore, by \cite[Proposition 3.3]{TiepKleshchev10}, there is some power of $q$, say $q^s$, such that $M:=HZ/Z$ satisfies one of the following:
\begin{enumerate}
\item $M\rhd SL_a(q^{s})$ with $q^{sa}=q^3$ for some $a\geq 2$
\item $M\rhd Sp_{2a}(q^{s})'$ with $q^{2sa}=q^3$ for some $a\geq 2$
\item $M\rhd G_2(q^s)'$ with $q^{6s}=q^3$,
or
\item $M\cdot (Z(GL_3(q)))\leq \Gamma L_1(q^3)$.
\end{enumerate}
Now, the conditions that $q^{2as}=q^3$ or $q^{6s}=q^3$ imply that $H$ cannot satisfy (2) or (3).  As $(q-1)(q^3-1)$ must divide $|M|$, $H$ also cannot satisfy (4).  Hence, $H$ is as in (1).  But then the conditions $q^{as}=q^3$ and $q\geq 2$ imply that $a=3$ and $s=1$.  Therefore, $SL_3(q)\lhd M=HZ/Z$.
\end{proof}

\begin{lemma}\label{lem:invsubgpsofZ}
A nontrivial $SL_3(q)$-invariant proper subgroup of $Z$ must be $R$.
\end{lemma}
\begin{proof}
Let $D< Z$ be nontrivial and invariant under the $SL_3(q)$-action, which is given by $XCX^T$ for $C\in Z$ and $X\in SL_3(q)$.  Note that here we have made the identifications $C\leftrightarrow
\left(
  \begin{array}{cc}
    I & C \\
    0 & I \\
  \end{array}
\right)$ and $X\leftrightarrow \left(
  \begin{array}{cc}
    X & 0\\
    0 & (X^{-1})^{T} \\
  \end{array}
\right)$.  Now, note that $SL_3(q)$ acts transitively on $R\setminus 0$, so $D\cap R$ must be either $R$ or $0$.  (Indeed, the action of $SL_3(q)$ on $R$ is the second wedge $\Lambda^2(U)\simeq U^\ast$ of the action on the natural module $U$ for $SL_3(q)$.)  Moreover, $SL_3(q)$ acts transitively on $(Z/R)\setminus 0$, so either $DR/R=Z/R$ or $DR=R$.   (Indeed, the action of $SL_3(q)$ on $Z/R$ is the Frobenius twist $U^{(2)}$ of the action of $SL_3(q)$ on the natural module $U$.)  If $R<D$, then $D/R=Z/R$, so $D=Z$, a contradiction.  Hence either $D=R$ or $D\cap R=0$.

If $D\cap R=0$, then $DR\neq R$, so $DR=Z$ and $D$ is a complement in $Z$ for $R$.  Hence no two elements of $D$ can have the same diagonal.  Let $g=\left(
                                                      \begin{array}{ccc}
                                                        1 & a & b \\
                                                        a & 0 & c \\
                                                        b & c & 0 \\
                                                      \end{array}
                                                    \right)$ be the element in $D$ with diagonal $(1,0,0)$, which must exist since $SL_3(q)$ acts transitively on nonzero elements of $DR/R=Z/R$.  If $g$ is diagonal, then any matrix of the form $\mathrm{diag}(a, 0, 0)$, $\mathrm{diag}(0, a, 0)$, or $\mathrm{diag}(0, 0, a)$ for $a\neq 0$ is in the orbit of $g$.  Thus since $D$ is an $SL_3(q)$-invariant subgroup, $D$ contains the group of all diagonal matrices.
As $D$ is a complement for $R$, it follows that in fact $D$ is the group of diagonal matrices, a contradiction since this group is not $SL_3(q)$-invariant.  Therefore, $g$ has nonzero nondiagonal entries. We claim that there is some $X\in SL_3(q)$ which stabilizes the coset $g+R$ but does not stabilize $g$. That is, $g$ and $XgX^T$ have the same diagonal, but are not the same element, yielding a contradiction.  Indeed, if at least one of $a, b$ is nonzero, then any $X= \mathrm{diag}(1, s, s^{-1})$ with $s \neq 1$ satisfies the claim.  If $a=b=0$ and $c\neq 0$, we can take $X$ to be $\left(
                                                      \begin{array}{ccc}
                                                        1 & r & 0 \\
                                                        0 & 1 & 0 \\
                                                        0 & 0 & 1 \\
                                                      \end{array}
                                                    \right)$ with $r\neq 0$, proving the claim.  We have therefore shown that $D=R$.
\end{proof}

\begin{theorem}
Let $G=Sp_6(q)$ with $q\geq 4$ even, and let $V$ be an absolutely irreducible $G$-module $V$ which affords the Brauer character $\wh{\alpha}_3$.  Then $V|_H$ is irreducible for some $H<P_3$ if and only if $H$ contains $P_3'=Z:SL_3(q)$.
\end{theorem}
\begin{proof}
First, if $H$ contains $P_3'$, then $V|_H$ is irreducible by \prettyref{lem:alphairredZSL}. Conversely, suppose that $V|_H$ is irreducible for some $H<P_3$.  Assume by way of contradiction that $H$ does not contain $P_3'$.  By \prettyref{lem:subgpcontainsSL}, $HZ$ contains $P_3'$, so $H\cap Z$ is $SL_3(q)$-invariant.  Therefore, by \prettyref{lem:invsubgpsofZ}, $H\cap Z$ must be $1, R,$ or $Z$.  Since $H$ does not contain $P_3'$, it follows that $H\cap Z=1$ or $R$.

Write $H_1:=H\cap P_3'$.  Then $H_1Z=P_3'$.  (Indeed, $P_3'\leq ZH$, so any $g\in P_3'$ can be written as $g=zh$ with $z\in Z, h\in H$.  Hence $z^{-1}g=h\in H\cap P_3'=H_1$, and $g\in H_1Z$.  On the other hand, $H_1Z\leq P_3'Z=P_3'$.)

Now, if $H\cap Z=1$, then $H_1\cap Z=1$ and $H_1\cong P_3'/Z=SL_3(q)$.  Since $V|_H$ is irreducible and $H/H_1$ is cyclic of order $q-1$, we see by Clifford theory that $H_1\cong SL_3(q)$ has an irreducible  character of degree $\alpha_3(1)/d$ for some $d$ dividing $q-1$.  Then $SL_3(q)$ has an irreducible character degree divisible by $q(q^3-1)/2$, a contradiction, as $\mathfrak{m}(SL_3(q))<q(q^3-1)/2$.

Hence we have $H\cap Z = R$.  Then $H_1\cap Z=R$ as well, so $(H_1/R)\cap (Z/R)=1$, and $H_1/R$ is a complement for $(Z/R)=[q^3]$ in $P_3'/R\cong[q^3]:SL_3(q)$. As the first cohomology group $H^1(SL_3(q), \F_q^3)$ is trivial (see, for example \cite[Table 4.5]{clineparshallscott}), any complement for $Z/R$ in $P_3'/R$ is conjugate in $P_3'/R$ to $H_1/R$.  In particular, writing $K_1:=R:L_3'= R:SL_3(q)$, we see that $K_1/R$ is also a complement for $Z/R$ in $P_3'/R$.  Hence, $K_1/R$ is conjugate to $H_1/R$ in $P_3'/R$, so $K_1$ is conjugate to $H_1$ in $P_3'$, and we may assume for the remainder of the proof that $H_1=R:L_3'=R:SL_3(q)$.

As $H/H_1$ is cyclic of order dividing $q-1$, we know by Clifford theory that if $\alpha_3|_H$ is irreducible, then there is some $d|(q-1)$ so that each irreducible constituent of $\alpha_3|_{H_1}$ has degree $\alpha_3(1)/d$.  Let $\beta\in\irr(H_1)$ be one such constituent, and let $\mu$ be a constituent of $\beta|_R$.  Then since $L_3'=SL_3(q)$ acts transitively on $\irr(R)\setminus\{1_R\}$, $I_{H_1}(\mu):=\stab_{H_1}(\mu)=R:([q^2]:SL_2(q))$. By Clifford theory, we can write $\beta|_{H_1}=\psi^{H_1}$ for some $\psi\in\irr(I_{H_1}(\mu)|\mu)$.  Hence $\beta(1)=[H:I_H(\mu)]\cdot\psi(1)=(q^3-1)\psi(1)$.

We can view $\psi$ as a character of $I_H(\mu)/\ker\mu$, as $\psi|_R=e\cdot\mu$ for some integer $e$.  But $I_H(\mu)/\ker\mu\cong C_2\times \left([q^2]:SL_2(q)\right)$, as $R$ is elementary abelian and $\mu$ is nontrivial.  If $\psi$ is nontrivial on $[q^2]$, then $\psi|_{[q^2]}$ is some integer times an orbit sum for some $SL_2(q)$-orbit of characters of $[q^2]$, again by Clifford theory.  However, as $SL_2(q)$ is transitive on $[q^2]\setminus0$, it follows that $\psi(1)$ is divisible by $q^2-1$, a contradiction since $\beta(1)$ is not divisible by $q^2-1$.

Hence $\psi$ is trivial on $[q^2]$, so $\psi$ can be viewed as a character of $C_2\times SL_2(q)$.  As $q\geq 4$, $\psi(1)=q(q-1)/2d$ is even. Now, the only even irreducible character degree of $SL_2(q)$ is $q$, but $q\neq q(q-1)/2d$, which contradicts the existence of this $\beta$.  Therefore, $\alpha_3|_H$ cannot be irreducible, so neither is $\wh{\alpha}_3|_H$.
\end{proof}

We have now completed the proof of \prettyref{thm:mainresult}.

%% file: PAPERVERSIONBrauerCharTableSp6q.tex
Tables \ref{tab:unipblockslnot3divq-1} through \ref{tab:unipblocksldivq^2+1} list the degrees and descriptions in terms of ordinary characters of the irreducible Brauer characters of $G=Sp_6(q)$, $q$ even, that lie in unipotent blocks for the various possibilities of $\ell||G|$, which can be extracted from \cite{white2000}.  We use the notation $\phi_i$ for the $i$th cyclotomic polynomial.  Also, $\alpha=2$ if $(q+1)_{\ell}\neq 3$ and is 1 if $(q+1)_{\ell}=3$.  The unknowns $\beta_i$, $i=2, 3$ satisfy $1\leq \beta_2\leq q/2+1,$ and $1\leq\beta_3\leq q/2$ (see \cite{white2000}).  Moreover, from \cite{white2000}, the unknown $\beta_1$ is either $0$ or $1$.  However, the results of \cite{TiepGuralnick04} yield that in fact $\beta_1=1$.

\begin{table}[ht]
\centering
\caption{$\ell-$Brauer Characters in Unipotent Blocks of $G=Sp_6(q)$, $\ell|(q-1)$, $\ell\neq 3$}\label{tab:unipblockslnot3divq-1}
\subfloat[][Principal Block $b_0$]{
\begin{tabular}{|c|c|}
  \hline
  $\varphi\in\ibr(G)\cap b_0$ & Degree, $\varphi(1)$\\
  \hline
  $\widehat{\chi}_1$ & 1\\
  $\widehat{\chi}_2$ & $\frac{1}{2}q(q^2+q+1)(q^2+1)$\\
  $\widehat{\chi}_3$ & $\frac{1}{2}q(q^2-q+1)(q^2+1)$\\
  $\widehat{\chi}_4$ & $\frac{1}{2}q(q^2-q+1)(q+1)^2$\\
  $\widehat{\chi}_6$ & $q^2(q^4+q^2+1)$\\
  $\widehat{\chi}_7$ & $q^3(q^4+q^2+1)$\\
  $\widehat{\chi}_8$ & $\frac{1}{2}q^4(q^2+q+1)(q^2+1)$\\
  $\widehat{\chi}_9$ & $\frac{1}{2}q^4(q^2-q+1)(q+1)^2$\\
  $\widehat{\chi}_{10}$ & $\frac{1}{2}q^4(q^2-q+1)(q^2+1)$\\
  $\widehat{\chi}_{12}$ & $q^9$\\
  \hline
 \end{tabular}}
\subfloat[][Block $b_1$]{
\begin{tabular}{|c||c|}
  \hline
  $\varphi\in\ibr(G)\cap b_1$ & Degree, $\varphi(1)$\\
  \hline
  $\widehat{\chi}_5$ & $\frac{1}{2}q(q-1)^2(q^2+q+1)$\\
  $\widehat{\chi}_{11}$ & $\frac{1}{2}q^4(q-1)^2(q^2+q+1)$\\
  \hline
  \end{tabular}}
\end{table}

\begin{table}[ht]
\centering
\caption{$\ell-$Brauer Characters in Unipotent Blocks of $G=Sp_6(q)$, $\ell=3|(q-1)$}\label{tab:unipblocksl3divq-1}
\subfloat[][Principal Block $b_0$]{
\begin{tabular}{|c|c|}
\hline
  $\varphi\in\ibr(G)\cap b_0$ & Degree, $\varphi(1)$\\
  \hline
  $\widehat{\chi}_1$ & 1\\
  $\widehat{\chi}_2$ & $\frac{1}{2}q(q^2+q+1)(q^2+1)$\\
  $\widehat{\chi}_3$ & $\frac{1}{2}q(q^2-q+1)(q^2+1)$\\
  $\widehat{\chi}_4-\widehat{\chi}_1$ & $\frac{1}{2}q(q^2-q+1)(q+1)^2-1$\\
  $\widehat{\chi}_6$ & $q^2(q^4+q^2+1)$\\
  $\widehat{\chi}_7$ & $q^3(q^4+q^2+1)$\\
  $\widehat{\chi}_8$ & $\frac{1}{2}q^4(q^2+q+1)(q^2+1)$\\
  $\widehat{\chi}_9-\widehat{\chi}_3$ & $\frac{1}{2}q^4(q^2-q+1)(q+1)^2-\frac{1}{2}q(q^2-q+1)(q^2+1)$\\
  $\widehat{\chi}_{10}-\widehat{\chi}_4+\widehat{\chi}_1$ & $\frac{1}{2}q^4(q^2-q+1)(q^2+1)-\frac{1}{2}q(q^2-q+1)(q+1)^2+1$\\
  $\widehat{\chi}_{12}-\widehat{\chi}_9+\widehat{\chi}_3$ & $q^9-\frac{1}{2}q^4(q^2-q+1)(q+1)^2+\frac{1}{2}q(q^2-q+1)(q^2+1)$\\
  \hline
  \end{tabular}}\qquad
\subfloat[][Block $b_1$]{
\begin{tabular}{|c|c|}
\hline
  $\varphi\in\ibr(G)\cap b_1$ & Degree, $\varphi(1)$\\
  \hline
  $\widehat{\chi}_5$ & $\frac{1}{2}q(q-1)^2(q^2+q+1)$\\
  $\widehat{\chi}_{11}$ & $\frac{1}{2}q^4(q-1)^2(q^2+q+1)$\\
  \hline
  \end{tabular}}
\end{table}

\begin{table}[ht]
\centering
\caption{$\ell-$Brauer Characters in Unipotent Blocks of $G=Sp_6(q)$, $\ell|(q+1)$}\label{tab:unipblocksldivq+1}
\subfloat[][Principal Block $b_0$]{
\begin{tabular}{|l|c|}
\hline
  $\varphi\in\ibr(G)\cap b_0$ & Degree, $\varphi(1)$\\
  \hline
  $\varphi_1=\wh{\chi}_1$ & 1\\
  $\varphi_2=\wh{\chi}_2-\beta_1\wh{\chi}_1$ & $\frac{1}{2}q(q^2+q+1)(q^2+1)-\beta_1$\\
  $\varphi_3=\wh{\chi}_3-\wh{\chi}_1$ & $\frac{1}{2}q(q^2-q+1)(q^2+1)-1$\\
  $\varphi_4=\wh{\chi}_5$ & $\frac{1}{2}q(q^2+q+1)(q-1)^2$\\
  $\varphi_5=\wh{\chi}_{28}$ & $(q^2+q+1)(q-1)^2(q^2+1)$\\
  $\quad =\wh{\chi}_6-\wh{\chi}_3-\wh{\chi}_2+\wh{\chi}_1$ & \\
  $\varphi_6=\wh{\chi}_{35}-\wh{\chi}_{5}$ & $\phi_1\phi_3(\phi_4\phi_6-\frac{1}{2}q\phi_1)$\\
  $\quad=\wh{\chi}_7-\wh{\chi}_6+\wh{\chi}_3-\wh{\chi}_1$& \\
  $\varphi_7=\wh{\chi}_{22}-(\alpha-1)\wh{\chi}_5-\wh{\chi}_3+\wh{\chi}_1$ & $\frac{1}{2}q\phi_1\phi_3\phi_4\phi_6-\frac{\alpha-1}{2}q\phi_1^2\phi_3-\frac{1}{2}q\phi_4\phi_6+1$\\
  $\quad=\wh{\chi}_8-\wh{\chi}_7-\alpha\wh{\chi}_5-\wh{\chi}_3+\wh{\chi}_1$ & \\
  $\varphi_8=\wh{\chi}_{23}$ & $\frac{1}{2}q\phi_1^3\phi_3\phi_6$\\
  $\quad = \wh{\chi}_{10}-\wh{\chi}_7+\wh{\chi}_6-\wh{\chi}_3$ & \\
  $\varphi_9=\wh{\chi}_{11}-\wh{\chi}_5$ & $\frac{1}{2}q\phi_1^3\phi_3^2$\\
  $\varphi_{10}=\wh{\chi}_{30}-\beta_3(\wh{\chi}_{11}-\wh{\chi}_5)-(\beta_2-1)\wh{\chi}_{23}-\wh{\chi}_{28}$ & $\phi_1^2\phi_3(q^3\phi_4-\frac{\beta_3}{2}q^4+ \frac{\beta_3}{2}q-\phi_4-\frac{\beta_2-1}{2}q\phi_1\phi_6) $\\
  \hline
\end{tabular}}\qquad
\subfloat[][Block $b_1$]{
\begin{tabular}{|l|c|}
  \hline
  $\varphi\in\ibr(G)\cap b_1$ & Degree, $\varphi(1)$\\
  \hline
  $\widehat{\chi}_4$ & $\frac{1}{2}q(q+1)^2(q^2-q+1)$\\
  $\widehat{\chi}_9-\widehat{\chi}_4$ & $\frac{1}{2}q(q+1)^2(q^2-q+1)(q^3-1)$\\
  \hline
\end{tabular}}
\end{table}

\begin{table}[ht]
\centering
\caption{$\ell-$Brauer Characters in Unipotent Blocks of $G=Sp_6(q)$, $\ell|(q^2-q+1)$, $\ell\neq 3$}\label{tab:unipblocksldivq^2-q+1}
\subfloat[][Principal Block $b_0$]{
\begin{tabular}{|c|c|}
  \hline
$\varphi\in\ibr(G)\cap b_0$ & Degree, $\varphi(1)$ \\
    \hline
    $\widehat{\chi}_1$ & 1 \\
    $\widehat{\chi}_2-\widehat{\chi}_1$ & $\frac{1}{2}q(q^2+q+1)(q^2+1)-1$ \\
    $\widehat{\chi}_8-\widehat{\chi}_2+\widehat{\chi}_1$ & $\frac{1}{2}q^4(q^2+q+1)(q^2+1)-\frac{1}{2}q(q^2+q+1)(q^2+1)+1$\\
    $\widehat{\chi}_{12}-\widehat{\chi}_8+\widehat{\chi}_2-\widehat{\chi}_1$ & $q^9-\frac{1}{2}q^4(q^2+q+1)(q^2+1)+\frac{1}{2}q(q^2+q+1)(q^2+1)-1$\\
    $\widehat{\chi}_5$ & $\frac{1}{2}q(q-1)^2(q^2+q+1)$\\
    $\widehat{\chi}_{11}-\widehat{\chi}_5$ & $\frac{1}{2}q^4(q-1)^2(q^2+q+1)-\frac{1}{2}q(q-1)^2(q^2+q+1)$\\
\hline
\end{tabular}}\qquad
\subfloat[][Blocks of Defect $0$ ]{
\begin{tabular}{|c|c|}
  \hline
$\varphi\in\ibr(G)$ & Degree, $\varphi(1)$ \\
\hline
$\widehat{\chi}_3$ & $\frac{1}{2}q(q^2-q+1)(q^2+1)$\\
$\widehat{\chi}_4$ & $\frac{1}{2}q(q^2-q+1)(q+1)^2$\\
$\widehat{\chi}_6$ & $q^2(q^4+q^2+1)$\\
$\widehat{\chi}_7$ & $q^3(q^4+q^2+1)$\\
$\widehat{\chi}_9$ & $\frac{1}{2}q^4(q^2-q+1)(q+1)^2$\\
$\widehat{\chi}_{10}$ & $\frac{1}{2}q^4(q^2-q+1)(q^2+1)$\\
\hline
\end{tabular}}
\end{table}

\begin{table}[ht]
  \centering
  \caption{$\ell-$Brauer Characters in Unipotent Blocks of $G=Sp_6(q)$, $\ell|(q^2+q+1)$, $\ell\neq 3$}\label{tab:unipblocksldivq^2+q+1}
\subfloat[][Principal Block $b_0$]{
  \begin{tabular}{|c|c|}
  \hline
    $\varphi\in\ibr(G)\cap b_0$ & Degree, $\varphi(1)$ \\
    \hline
    $\widehat{\chi}_1$ & 1 \\
    $\widehat{\chi}_4-\widehat{\chi}_1$ & $\frac{1}{2}q(q+1)^2(q^2-q+1)-1$\\
    $\widehat{\chi}_{10}-\widehat{\chi}_4+\widehat{\chi}_1$ & $\frac{1}{2}q^4(q^2+1)(q^2-q+1)-\frac{1}{2}q(q+1)^2(q^2-q+1)+1$ \\
    $\widehat{\chi}_3$ & $\frac{1}{2}q(q^2+1)(q^2-q+1)$\\
    $\widehat{\chi}_9-\widehat{\chi}_3$ & $\frac{1}{2}q^4(q+1)^2(q^2-q+1)-\frac{1}{2}q(q^2+1)(q^2-q+1)$\\
    $\widehat{\chi}_{12}-\widehat{\chi}_9+\widehat{\chi}_3$ & $q^9-\frac{1}{2}q^4(q+1)^2(q^2-q+1)+\frac{1}{2}q(q^2+1)(q^2-q+1)$\\
    \hline
  \end{tabular}}\qquad
\subfloat[][Blocks of Defect $0$]{
\begin{tabular}{|c|c|}
\hline
$\varphi\in\ibr(G)$ & Degree, $\varphi(1)$ \\
\hline
$\widehat{\chi}_2$ & $\frac{1}{2}q(q^2+q+1)(q^2+1)$\\
$\widehat{\chi}_5$ & $\frac{1}{2}q(q-1)^2(q^2+q+1)$\\
$\widehat{\chi}_6$ & $q^2(q^4+q^2+1)$\\
  $\widehat{\chi}_7$ & $q^3(q^4+q^2+1)$\\
  $\widehat{\chi}_8$ & $\frac{1}{2}q^4(q^2+q+1)(q^2+1)$\\
$\widehat{\chi}_{11}$ &$\frac{1}{2}q^4(q-1)^2(q^2+q+1)$ \\
\hline
  \end{tabular}}
\end{table}

\begin{table}[ht]
\centering
\caption{$\ell-$Brauer Characters in Unipotent Blocks of $G=Sp_6(q)$, $\ell|(q^2+1)$}\label{tab:unipblocksldivq^2+1}
\subfloat[][Principal Block $b_0$]{
\begin{tabular}{|l|c|}
\hline
   $\varphi\in\ibr(G)\cap b_0$ & Degree, $\varphi(1)$\\
  \hline
  $\wh{\chi}_1$ & 1\\
  $\wh{\chi}_6-\wh{\chi}_1$ & $q^2(q^4+q^2+1)-1$\\
  $\wh{\chi}_9-\wh{\chi}_6+\wh{\chi}_1$ & $\frac{1}{2}q^4(q+1)^2(q^2-q+1)-q^2(q^4+q^2+1)+1$\\
  $\wh{\chi}_{11}$ & $\frac{1}{2}q^4(q-1)^2(q^2+q+1)$\\
  \hline
\end{tabular}}\qquad
\subfloat[][Block $b_1$]{
\begin{tabular}{|l|c|}
 \hline
  $\varphi\in\ibr(G)\cap b_1$ & Degree, $\varphi(1)$\\
  \hline
  $\widehat{\chi}_4$ & $\frac{1}{2}q(q+1)^2(q^2-q+1)$\\
  $\widehat{\chi}_7-\widehat{\chi}_4$ & $q^3(q^4+q^2+1)-\frac{1}{2}q(q+1)^2(q^2-q+1)$\\
  $\wh{\chi}_{12}-\wh{\chi}_7+\wh{\chi}_4$ & $q^9-q^3(q^4+q^2+1)+\frac{1}{2}q(q+1)^2(q^2-q+1)$\\
  $\wh{\chi}_5$ & $\frac{1}{2}q(q-1)^2(q^2+q+1)$\\
  \hline
\end{tabular}}\qquad
\subfloat[][Blocks of Defect $0$]{
\begin{tabular}{|l|c|}
\hline
$\varphi\in\ibr(G)$ & Degree, $\varphi(1)$ \\
\hline
$\widehat{\chi}_2$ & $\frac{1}{2}q(q^2+q+1)(q^2+1)$\\
$\widehat{\chi}_3$ & $\frac{1}{2}q(q^2+1)(q^2-q+1)$\\
$\widehat{\chi}_8$ & $\frac{1}{2}q^4(q^2+q+1)(q^2+1)$\\
$\widehat{\chi}_{10}$ & $\frac{1}{2}q^4(q^2+1)(q^2-q+1)$\\
\hline
\end{tabular}}
\end{table} 

%% file: PAPERVERSIONNotationsCharsSpG2.tex
\begin{table}
\centering
\caption{Notation of Characters of $Sp_6(q)$}
\label{tab:notationSp}
\begin{tabular}{|c|c|c|c|}
\hline
Degree &Guralnick-Tiep \cite{TiepGuralnick04} & Luebeck \cite{Luebeckthesis} & D. White \cite{white2000} \\
\hline
$\frac{(q^3+1)(q^3-q)}{2(q-1)}$&$\rho_3^1$&$\chi_{1,4}$&$\chi_4$\\
\hline
$\frac{(q^3-1)(q^3+q)}{2(q-1)}$&$\rho_3^2$&$\chi_{1,2}$&$\chi_2$\\
\hline
$\frac{q^6-1}{q-1}$&$\tau_3^i $& &Type $\chi_{13}$\\
\hline
$\frac{(q^3-1)(q^3-q)}{2(q+1)}$&$\alpha_3$& $\chi_{1,5}$&$\chi_5$\\
\hline
$\frac{(q^3+1)(q^3+q)}{2(q+1)}$&$\beta_3$&$\chi_{1,3}$&$\chi_3$\\
\hline
$\frac{q^6-1}{q+1}$&$\zeta_3^i $&&Type $\chi_{19}$\\
\hline
\end{tabular}
\end{table}

\begin{table}
\centering
\caption{Notation of Characters of $G_2(q)$}
\label{tab:notationG2}
\begin{tabular}{|c|c|c|c|}
\hline
Degree & Guralnick-Tiep \cite{TiepGuralnick04} &Enomoto-Yamada \cite{EnomotoYamada} & Hiss-Shamash \cite{Hiss89}, \cite{HissShamash90},\cite{shamash87},\cite{shamash89},\cite{shamash92}\\
\hline
$\frac{(q^3+1)(q^3-q)}{2(q-1)}$&$(\rho_3^1)|_{G_2(q)}$&$\theta_2$&$X_{15}$\\
\hline
$\frac{(q^3-1)(q^3-q)}{2(q+1)}$&$(\alpha_3)|_{G_2(q)}$&$\theta_2'$&$X_{17}$\\
\hline
$\frac{q^6-1}{q-1}$&$(\tau_3^i)|_{G_2(q)}$&$\chi_3(i)$&$X_{1b}'$\\
\hline
$\frac{q^6-1}{q+1}$&$(\zeta_3^i)|_{G_2(q)}$&$\chi_3'(i)$&$X_{2a}'$\\
\hline
$\frac{q(q^2+q+1)(q+1)^2}{6}$& & $\theta_1$ & $X_{16}$\\
\hline
$\frac{q(q^2-q+1)(q-1)^2}{6}$& & $\theta_1'$ & $X_{18}$\\
\hline
$\frac{q(q^4+q^2+1)}{3}$& & $\theta_4$ &$X_{14}$\\
\hline

\end{tabular}
\end{table}